\documentclass[11pt, a4paper]{amsart}
\usepackage{amsfonts,amssymb,amsmath, amsthm}
\usepackage{mathrsfs, mathtools}
\usepackage{lmodern}
\usepackage{qsymbols}
\usepackage{latexsym}
\usepackage[noadjust]{cite}
\usepackage{pdfsync}
\usepackage{bm}
\usepackage{enumitem}

\newtheorem{thm}{Theorem}[section]
\newtheorem{lem}[thm]{Lemma}
\newtheorem{prop}[thm]{Proposition}

\theoremstyle{definition}
\newtheorem{defn}[thm]{Definition}
\newtheorem{rem}[thm]{Remark}
\newtheorem{exm}[thm]{Example}

\numberwithin{equation}{section}
\usepackage[plainpages=false,pdfpagelabels,backref=page,citecolor=red]{hyperref}
\usepackage{xcolor}
\hypersetup{
colorlinks,
linkcolor={cyan!90!black},
citecolor={magenta},
urlcolor={green!40!black}
} 
\newcommand{\R}{\mathbb{R}}

\newcommand{\loc}{\operatorname{loc}}
\renewcommand{\L}{\operatorname{L}} 
\renewcommand{\H}{\operatorname{H}}

\renewcommand{\H}{\operatorname{H}} 
\newcommand{\W}{\operatorname{W}}

\renewcommand{\div}{\operatorname{div}}

\renewcommand{\d}{\, \mathrm{d}} 

\DeclareMathOperator{\supp}{supp} 
\newcommand{\sgn}{\operatorname{sgn}}
\def\Xint#1{\mathchoice
{\XXint\displaystyle\textstyle{#1}}%
{\XXint\textstyle\scriptstyle{#1}}%
{\XXint\scriptstyle\scriptscriptstyle{#1}}%
{\XXint\scriptscriptstyle%
\scriptscriptstyle{#1}}%
\!\int}
\def\XXint#1#2#3{{\setbox0=\hbox{$#1{#2#3}{%
\int}$ }
\vcenter{\hbox{$#2#3$ }}\kern-.6\wd0}}
\def\barint{\,\Xint -} 
\def\bariint{\barint_{} \kern-.4em \barint}
\def\bariiint{\bariint_{} \kern-.4em \barint}
\renewcommand{\iint}{\int_{}\kern-.34em \int} 
\renewcommand{\iiint}{\iint_{}\kern-.34em \int} 

\author{Alireza Ataei}
\email{alireza.ataei@math.uu.se}
\address{Department of Mathematics, Uppsala University, S-751 06 Uppsala,
Sweden}

\thanks{}


\begin{document}
\title{Existence and uniqueness of the solutions to Convection-Diffusion equations}

\begin{abstract}
    In this work, we study convection-diffusion equations in the cases of bounded drifts and drifts induced by the gradient of a potential.
    We define a new notion of solution and prove its existence and uniqueness. Furthermore, we show the conservation of mass, the convergence to the initial data, and the strong maximum principle.   
\end{abstract}

\maketitle

\section{Introduction and the statements of the results}
Convection-diffusion equations appear in many parts of science describing several phenomena, see \cite{FG,G,R,S,T.S,SJ}. The general prototype of convection-diffusion equations is 
\begin{align*}
    \partial_t u - \div A \nabla u + \div(v \, u) = R,
\end{align*}
with initial data $g$, where $g\in \L^1(\R^n), R \in \L^1([0,T] \times \R^n)$, $A=(a_{ij})_{i,j=1}^n$ is an elliptic matrix-valued function, see \eqref{eq:ellipticcond} for more detail, and $v$ is either bounded or is of the form $v = -\nabla V$ where $e^V+e^{-V} \in \L^1_{\loc}(\R^n)$ and $e^{-V}$ is a regular weight, which means that for every function $f \in \W^{1,1}_{\loc}(\R^n)$ satisfying 
\begin{align*}
   \|f\|^2_{\H^1_{e^{-V}}(\R^n)} := \int_{\R^n} \big(|f|^2 +|\nabla f|^2\big) \, e^{-V} \d x < \infty,
\end{align*}
there exists a sequence $f_i \in C^{\infty}_0(\R^n)$ such that 
\begin{align*}
  \lim_{i \to \infty} \|f_i-f\|_{\H^1_{e^{-V}}}=0.
\end{align*}
The existence and uniqueness of solutions for such equations are not yet fully understood.

There are several ways to define solutions if the drift $v$ is not smooth, namely very weak solutions, weak solutions, entropy solutions, and renormalized solutions. The weak solution $u \in \L^1([0,T];\W^{1,1}_{\loc}(\R^n))$ satisfies
\begin{align}
    \label{eq:weaksolconv1}
     \int_0^T \int_{\R^n} -u \, \partial_t \phi +A \nabla u \cdot \nabla \phi - u \,  v \cdot \nabla \phi \, \d x \d t =  \int_{\R^n} g \, \phi\, \d x \bigg |_0 + \int_0^T \int_{\R^n} R \, \phi \, \d x,
\end{align}
 for every $\phi \in C^{\infty}_0([0,T] \times \R^n)$ with $\phi(T,x)=0$ for every $x \in \R^n$. Moreover, assuming that
 $A$ satisfies 
\begin{align}
\label{eq:derviassum}
    \sum_{i=1}^n \frac{\partial a_{ij}}{\partial x_i} \in \L^{\infty}([0,T] \times \R^n), 
\end{align}
for all $1 \leq j \leq n$, then a very weak solution $u \in \L^1([0,T];\L^1_{\loc}(\R^n))$ is defined by the property
\begin{align}
\label{eq:weaksolconv2}
  \int_0^T \int_{\R^n} u \, (-\partial_t \phi + \div A^* \nabla \phi - v \cdot \nabla \phi) \, \d x \d t =  \int_{\R^n} g \, \phi\, \d x \bigg |_0 + \int_0^T \int_{\R^n} R \, \phi \, \d x.
\end{align} for every $\phi \in C^{\infty}_0([0,T] \times \R^n)$ with $\phi(T,x)=0$ for every $x \in \R^n$. If the Aronson-Serrin condition holds, see \cite{AS}, then one can use regularity estimates in \cite{AS} to derive the existence of weak solutions. In the case that $v$ and $\div v$ are bounded, $R=0$, and \eqref{eq:derviassum}, see \cite{F}, it is proved that a unique very weak solution exists, see also \cite[Thm. 6.6.2]{BKRS} for the case of replacing $\div A \nabla u$ with $ \sum_{i,j=1}^n\frac{\partial}{\partial x_i} \frac{\partial}{\partial x_j}(a_{ij} u)$ which does not require boundedness condition on $\div v$. In general case, however, there are some restrictions to the notion of weak and very weak solutions. First, it makes sense only if the product $u \, v$ is locally integrable, which is unclear if $v= -\nabla V$ where $V: \R^n \to \R$ does not belong to a Sobolev space. The second and perhaps more crucial issue is uniqueness. If one removes the assumption of boundedness for $v$ or \eqref{eq:derviassum}, then there are counterexamples of uniqueness. For example, in \cite[Thm. 1.2]{MS}, it is proved that there are $\L^p$-bounded drifts $v$ satisfying $\div v=0$ with infinitely many weak solutions satisfying \eqref{eq:weaksolconv1} in the periodic case, replacing $\R^n$ with $\mathbb{T}^n$. In the case of $R=0, A= I,$ and smooth drifts $v$, it is necessary to have a special growth bound on $v$, see \cite[Ch. 9]{BKRS}, otherwise there are infinitely many very weak solutions for any initial data which is a probability density function, see \cite[Thm. 1]{BKS}. Moreover, in Example \ref{exm:example}, motivated by \cite{Prignet,Serrin}, we prove the lack of uniqueness if \eqref{eq:derviassum} does not hold. To resolve the issues of existence and uniqueness,
the notion of entropy solutions was introduced in \cite{Prignet,BOP} for nonlinear parabolic equations, see also \cite{BBGGPV} for nonlinear elliptic equations. The advantages are that the entropy solutions exist even for square-integrable drifts, and whenever $v=0$ they are unique. However, if the drift is non-zero, the uniqueness and the connection to weak solutions are not clear. Finally, the notion of the renormalized solutions was defined to address some of the previous issues, see the original idea for transport equations in  \cite{DL} and then for the Fokker-Planck equations in \cite{LL}. However, for the existence and uniqueness, one requires certain properties such as $\div v$ being bounded, see \cite{LL,LL1}. We also mention that in \cite{P2} it is proven that if a weak solution $u$ exists for $A=I$ and $|v|^2 \, |u|$ is integrable, then the weak solution is unique and it is a renormalized solution (although the proof is for bounded sets one can generalize it to $\R^n$). 

Motivated by \cite{DL,P2,P3}, we define the following notion of solutions, see Definition \ref{def:convec1} and Definition \ref{def:convec2} for more details: A convection solution $u$ of 
\begin{align*}
    \partial_t u - \div A \nabla u + \div(v \,u) = R,
\end{align*}
in $[0,T] \times \R^n$ with initial data $g$, satisfies 
\begin{align}
\label{eq:convecequation}
     \int_0^T \int_{\R^n} u\, f\, \d x \d t =\int_{\R^n} g\, \phi \, \d x \bigg |_0 + \int_{0}^T \int_{\R^n} R\, \phi \, \d x \d t,
\end{align}
for every $f \in C^{\infty}_0([0,T] \times \R^n)$ and $\phi$ satisfying the adjoint equation $$-\partial_t \phi -\div A^* \nabla \phi - v \cdot \nabla \phi = f,$$ in $[0,T] \times \R^n$ with $\phi(T,x)=0$ for a.e. $x \in \R^n$, where $A^*$ is the transpose of $A$. We prove, by  Proposition \ref{prop:dualsol} and Proposition \ref{prop:advecsol}, that the adjoint equation has a unique solution. The benefits of working with the adjoint equation $-\partial_t \phi -\div A^* \nabla \phi - v \cdot \nabla \phi = f$ are the maximum principle property and energy estimates, see Proposition \ref{prop:dualsol} and Proposition \ref{prop:advecsol}, which implies the unique existence in a certain class of functions. The definition of convection solution is a generalization of very weak solutions by setting $-\partial_t \phi -\div A^* \nabla \phi - v \cdot \nabla \phi =f$ in \eqref{eq:weaksolconv2}. However, the uniqueness property and maximum principle are immediate from the definition of convection solutions, which is not the case for weak solutions.

In the first part of the paper, we focus on bounded drifts and prove the following: Let $A$ be an elliptic measurable matrix-valued function, see \eqref{eq:ellipticcond}. Then, the following holds.
\begin{thm}
\label{thm:boundeddrift}
    Let $v \in \L^{\infty}([0,T] \times \R^n; \R^n), g \in \L^1(\R^n),$ and $R \in \L^1([0,T] \times \R^n) $. Then, there exists a unique convection solution $u \in \L^1([0,T] \times \R^n) \cap \L^p([0,T];\W^{1,p}_{\loc}(\R^n))$, for every $1 \leq p<\frac{n}{n-1}$, of \begin{align*}
    \partial_t u - \div A \nabla u + \div(v\, u) = R,
\end{align*}
in $[0,T] \times \R^n$ with initial data $g$. Moreover, $u$ satisfies the following  properties:\\\\
(i). Let $\varphi_{\epsilon}$ be a non-negative smooth approximation of the identity and $u_{\epsilon} \in \L^2([0,T];\H^1(\R^n))$ be the smooth weak solution of \begin{align*}
    \partial_t u_{\epsilon} - \div (\varphi_{\epsilon}\ast A) \nabla u_{\epsilon} + \div((\varphi_{\epsilon}\ast v)\, u_{\epsilon}) = \varphi_{\epsilon}\ast R,
\end{align*}
in $[0,T] \times \R^n$ with $u_{\epsilon}(0,x)= \varphi_{\epsilon}\ast g (x)$ for $x \in \R^n$. Then, 
\begin{align*}
    \lim_{\epsilon \to 0} \int_{0}^T \int_{\R^n} u_{\epsilon} \, \psi \, \d x \d t =\int_{0}^T \int_{\R^n} u \, \psi \, \d x \d t, 
\end{align*}
for every $\psi \in C^{\infty}_0([0,T] \times \R^n).$\\
(ii). There exists $F \subset [0,T]$ of measure zero such that 
\begin{align*}
    \lim_{s \to 0; s \in [0,T] \setminus F} \int_{\R^n} u \, \psi \, \d x \bigg |_s = \int_{\R^n} g \, \psi \, \d x,
\end{align*}
for every bounded continuous function $\psi: \R^n \to \R$.\\
(iii). If $R,g$ are non-negative, then $u$ is non-negative. Furthermore, if $R, g$ are non-negative and $(\ln(g))^- \in \L^1_{\loc}(\R^n)$, then $u>0$ a.e. in $[0,T] \times \R^n$, $\ln u \in \L^2([0,T]; W^{1,2}_{\loc}(\R^n))$, and 
\begin{equation}
\begin{aligned}
\label{eq:strongmaxprin}
    &\int_0^T \int_{\R^n} -\ln u \, \partial_t \varphi  \, \d x \d t+ \int_0^T \int_{\R^n}A \nabla (\ln u) \cdot \nabla \varphi \, \d x \d t - \int_0^T \int_{\R^n}\varphi \, A \nabla (\ln u)  \cdot \nabla (\ln u)  \, \d x \d t\\& -\int_0^T \int_{\R^n} v \cdot \nabla \varphi  \d x \d t + \int_0^T \int_{\R^n}\varphi \, v \cdot \nabla (\ln u)  \d x \d t \geq \int_0^T \int_{\R^n} \frac{R}{u} \, \varphi \, \d x \d t,
\end{aligned}
\end{equation}
for every non-negative $\varphi \in C^{\infty}_0((0,T) \times \R^n).$ \\
(iv). The convection solution is a weak solution of $\partial_t u - \div A \nabla u +\div(v\, u)=R$ with initial data $g$.
\\
(v). If \eqref{eq:derviassum} holds, then the convection solution satisfies the following conservation of mass:
\begin{align*}
    \int_{\R^n} u \, \d x \bigg |_s = \int_{\R^n} g \, \d x + \int_0^s \int_{\R^n} R \, \d x \d t,
\end{align*}
for a.e. $s \in [0,T].$
\end{thm}
To explain the theorem, property (i) shows that any sequence of solutions to the equations obtained by the regularization of data converges vaguely to the unique convection solution. In property (ii), the weak convergence to initial data $g$ is proved, which is stronger than the vague convergence up to a subset depending on $\psi$ that was the strongest result proven so far, see \cite[Thm. 6.6.2]{BKRS}. Property (iii) derives both the weak and strong maximum principle for convection solutions. Finally, properties (iv) and (v) demonstrate that the convection solution is a generalization of the weak solution and satisfies the conservation of mass if \eqref{eq:derviassum} holds, respectively.

The second part of this work focuses on the case of $A = I, v= \nabla V$, where $V: \R^n \to \R$ is a measurable function satisfying $e^V+e^{-V} \in \L^1_{\loc}(\R^n)$ and $e^{-V}$ is a regular weight. Hence, we do not assume any integrability condition on $\nabla V$ or decay estimates on $V$ at infinity since several interesting applications require considering such a large class of potentials. Some key examples are $V \in \L_{\loc}^{\infty}(\R^n)$ and $V(x) = C \ln|x|$ for $x \in \R^2$ where $-2<C<2$, see \cite{DLBK,S} and references therein. However, the literature on existence and uniqueness is restricted to potentials in Sobolev spaces, see \cite{JKO} and \cite[Thm. 2.2.29]{Royer}. The idea in \cite{JKO} is to use a minimization technique that requires regularity and decay assumptions on the potential. The other approach in \cite[Thm. 2.2.29]{Royer} is to show that $h= u \, e^{V}$ is the weak solution of $\partial_t h - \Delta h + \nabla V \cdot \nabla h=R \, e^{V}$ and then using regularity estimates to derive solutions. However, if $V$ is not regular enough, then the transformation idea does not work. To define a well-defined notion of solution, we again define the adjoint convection solution $\phi$ for 
$-\partial_t \phi - \Delta \phi + \nabla V \cdot \nabla \phi =f$ in $[0,T] \times \R^n$ for every $f \in C^{\infty}_0([0,T] \times \R^n)$ with $\phi(T,x)=0$ for a.e. $x \in \R^n$. It is well-known that formally $\phi$ satisfies the weighted backward parabolic equation 
\begin{align}
\label{eq:weghtbackwardequ}
    -\partial_t \phi - e^V \div (e^{-V} \nabla \phi) =f,
\end{align}
in $[0,T] \times \R^n$ with $\phi(T,x)=0$ for a.e. $x \in \R^n.$ We prove in Proposition \ref{prop:advecsol} that the equation \eqref{eq:weghtbackwardequ} has a unique solution. Now, we denote $\mathcal{M}([0,T] \times \R^n)$ as all the bounded signed Borel measures. For every $R \in \L^1([0,T] \times \R^n), g \in \L^1(\R^n)$, we define a convection solution $\mu \in \mathcal{M}([0,T] \times \R^n)$ of
\begin{align*}
    \partial_t \mu - \Delta \mu - \div(\mu \, \nabla V) = R,
\end{align*}
in $[0,T] \times \R^n$ with initial data $g$ if 
\begin{align*}
 \int_{0}^{T}\int_{\R^n} f \, \d \mu = \int_0^T \int_{\R^n} g \,\phi \, \d x \bigg |_0 + \int_{0}^T \int_{\R^n} R \, \phi \, \d x \d t,
\end{align*}
for every $f \in C^{\infty}_0([0,T] \times \R^n)$ and adjoint convection solution $\phi$ of  $-\partial_t \phi - e^V \div (e^{-V} \nabla \phi) =f$ in $[0,T] \times \R^n$ with $\phi(T,x)=0$ for a.e. $x \in \R^n$. 

Finally, we prove the following:
\begin{thm}\label{thm:potentialdrift}
    Let $g \in \L^1(\R^n)$, $R \in \L^1([0,T] \times \R^n)$. Then, there exists a unique convection solution $\mu \in \mathcal{M}([0,T] \times \R^n)$ of
 \begin{align*}
     \partial_t \mu -\Delta \mu- \div(\mu \nabla V) = R,
 \end{align*}
in $[0,T] \times \R^n$ with initial data $g$, where
\begin{align*}
    |\mu|([0,T] \times \R^n) \leq T \int_{\R^n} |g| \, \d x +T \int_{0}^T \int_{\R^n} |R| \, \d x \d t.
\end{align*}
Furthermore, $\mu$ satisfies the following:\\
(i). Let $\varphi_{\epsilon}$ be a positive smooth approximation of the identity and $V_{\epsilon} \in C^{0,1}(\R^n) \cap C^{\infty}(\R^n)$ be chosen such that $e^{V_{\epsilon}},e^{-V_{\epsilon}}$ converge to $e^V,e^{-V}$ in $\L^1_{\loc}(\R^n)$, respectively (see Lemma \ref{lem:approxpoten} for the existence of such a sequence). Assume that $\mu_{\epsilon} \in \L^2([0,T]; \H^1(\R^n))$ is the smooth weak solution of 
\begin{align*}
    \partial_t \mu_{\epsilon} - \Delta \mu_{\epsilon} - \div( \mu_{\epsilon} \nabla V_{\epsilon}) = \varphi_{\epsilon} \ast R,
\end{align*}
in $[0,T] \times \R^n$ with $\mu_{\epsilon}(0,x)=\varphi_{\epsilon} \ast g(x)$ for every $x \in \R^n.$ Then, 
\begin{align*}
    \lim_{\epsilon \to 0} \int_{0}^T \int_{\R^n} \psi \, \mu_{\epsilon} \, \d x \d t =\int_{0}^T \int_{\R^n}  \psi \, \d \mu, 
\end{align*}
for every $\psi \in C^{\infty}_0([0,T] \times \R^n).$\\\\
(ii). If $R,g$ are non-negative, then $\mu$ is a positive measure and 
\begin{align*}
 \liminf_{r \to 0} \frac{\mu((s-r,s+r) \times B(x,r))}{|(s-r,s+r) \times B(x,r)|}  \geq (\inf  g\, e^V  + s \inf R\, e^V)  e^{-V(x)},
\end{align*}
for every $s \in (0,T)$ and every Lebesgue point $x \in \R^n$ of $e^{-V}.$ 
\\\\
(iii). If \begin{align*}
 \int_{\R^n} |g|^{q'} e^{\frac{q'}{q}V} \, \d x   + \int_0^T \biggl(\int_{\R^n} |R|^{p'} e^{\frac{p'}{p}V} \, \d x \biggr)^{\frac{1}{p'}}\d t < \infty,
\end{align*}
for some $p>1,q >1$, then $\mu \in \L^1([0,T] \times \R^n)$.\\\\
(iv). If
\begin{align}\label{eq:boundednes}
   \| g \, e^{V}\|_{\L^{\infty}} + \|R \, e^{V} \|_{\L^{\infty}} < \infty,
\end{align}
then $\mu \in \L^1([0,T]\times \R^n)$ and
\begin{align*}
    \| \mu \, e^{V} \|_{\L^{\infty}} \leq  \| g\, e^V \, \|_{\L^{\infty}} + T\,
    \|R \, e^V \|_{\L^{\infty}}. 
\end{align*}
\\\\
(v).  If $e^{-V} \in \L^1(\R^n)$, then the  signed measure $\nu(I) := \mu(I \times \R^n)$, for Borel subsets $I \subset [0,T]$, is absolutely continuous with respect to Lebesgue's measure $\mathcal{L}^1$ on $[0,T]$, and $\mu$ satisfies the following conservation of mass:
\begin{align*}
    \frac{\d \nu}{\d \mathcal{L}^1}(s) =  \int_{\R^n} g \, \d x + \int_0^s \int_{\R^n} R \, \d x \d t,
\end{align*}
for a.e. $s \in [0,T].$\\
(vi). For every $f \in C^{\infty}_0(\R^n)$, the following convergence to the initial data holds:
\begin{align*}
         \lim_{(h,s) \to 0}  \frac{1}{h} \int_s^{s+h} \int_{\R^n} f \, d \mu =  \int_{\R^n} f\, g \, \d x  .
\end{align*}
where in the limit we let $(h,s)$ goes to zero in the domain $0<h+s <T$.
\end{thm}
To elaborate on the result, part (i) proves that any good regularization on the data, see Lemma \ref{lem:approxpoten}, gives a sequence of solutions that converges vaguely to the convection solution. Parts (ii) and (iii) are a generalized form of the strong maximum principle and a sufficient condition for having $\L^1$-bounded solutions, respectively. In part (iv), we prove an upper bound estimate for the convection solution if \eqref{eq:boundednes} holds. Then, with the assumption $e^{-V} \in \L^1(\R^n)$, we demonstrate the conservation of mass. Finally, the last part proves the convergence to the initial data in the vague sense.  

We briefly discuss the new techniques in the proofs for Theorem \ref{thm:boundeddrift} and Theorem \ref{thm:potentialdrift}. For the first one, we apply the argument in \cite[Lem. 3.4]{P2}, together with using regularization, to derive the unique adjoint solution for $-\partial_t \phi -\div A^* \nabla \phi - v \cdot \nabla \phi = f$ for bounded and square-integrable functions $f$ in $\R^n.$ We then prove the conservation of mass for regular solutions and use it to derive a sequence of functions converging vaguely to the solution $u$ and prove continuous convergence to initial data. To prove the strong maximum principle, however, a nonlinear argument is applied. We first use the linearity and perturbation to consider the strictly positive approximating solutions $u_i$, and then extract the properties of the equation for $\ln(u_i)$. This approach is new and uses a property of nonlinear parabolic equations to prove the strong maximum principle. For the second result, we need to use new arguments for all the steps. First, the notion of the adjoint convection solutions is defined on a suitable weighted Sobolev space. Then, uniqueness is followed by a test function argument, and existence is obtained by approximating solutions and their regularity estimates, see Proposition \ref{prop:advecsol}. Then, again using the conservation of mass for the approximating smooth solutions, we derive the convection solution as a vague limit. For the upper and lower bounds, we need to use a weighted $\L^1$-estimate on the adjoint convection solution, see Lemma \ref{lem:upperL1bound}, together with the Radon-Nikodym theorem. In the next part, if $e^{-V} \in \L^1(\R^n)$, then, by Prokhorov's theorem, we derive the weak convergence of approximating solutions to convection solutions. This derives the conservation of mass for the convection solutions. Finally, we use properties of the time derivative of adjoint convection solutions, see Lemma \ref{lem:timederiv}, to obtain the continuous convergence to the initial data.

\section{Acknowledgement}
I want to thank Kaj Nyström for his encouragement for this work. Moreover, I thank Benny Avelin for motivating discussions and comments on this paper. The author also acknowledges a communication with Pierre-Louis Lions and appreciates his comments.

\section{Preliminaries and notations}
In the whole note, $n$ is a positive integer number. The set $B(x,r)$, for $x \in \R^n, r>0$, denotes the ball of radius $r$ centered at $x$. We define the dilation $\epsilon\, B(x,r) = B(x,\epsilon\, r)$ for every $\epsilon>0,r>0,x \in \R^n.$ The Hölder dual of a real $p>1$ is denoted by $p':= \frac{p}{p-1}.$

\subsection{Spaces and measures} \label{sec:spaceandmeas}
Let $T>0$ and $(\mathcal{B},\|\,\|)$ be a normed vector space. For every $F \subset [0,T], t \in [0,T]$, and $f: [0,T] \to \R$, we say that 
\begin{align*}
    \limsup_{s \to t;s \in F} f(s)=g,
\end{align*}
for $g \in \R$ if $$\limsup_{n \to \infty} f(s_n)=g,$$ for every sequence $s_n \in F$ such that $\lim_{n \to \infty} s_n = t.$ Likewise, we define $ \liminf_{s \to t;s \in F}$ and $\lim_{s \to t;s \in F}.$ A function $f: [0,T] \to \mathcal{B}$ is measurable if $f^{-1}(U)$ is Lebesgue measurable for every open subset $U \subset \mathcal{B}.$  Let $p,q$ be positive real numbers. Then, $\L^p([0,T];\mathcal{B})$ denotes the set of all the measurable functions $f: [0,T] \to \mathcal{B}$ such that 
\begin{align*}
    \|f\|^p_{\L^p([0,T];\mathcal{B})} := \int_{0}^T  ||f(t)||^p \, \d t < \infty,
\end{align*}
and $\L^{\infty}([0,T];\mathcal{B})$ includes all the measurable functions $g: [0,T] \to \mathcal{B}$ such that
\begin{align*}
    \|g\|_{\L^{\infty}([0,T];\mathcal{B})} := \sup_{t \in [0,T]} \|g(t)\| < \infty. 
\end{align*} 
For every $f:[0,T] \to \mathcal{B}$ and $h >0$, we define the shift operator $\tau_h f : [-h,T-h] \to \mathcal{B}$ by $\tau_h f(t) =f(t+h)$ for every $t \in [-h,T-h].$ The space $\L^q([0,T];W_{\loc}^{k,p}(\R^n))$ includes all the measurable real valued functions $f: [0,T] \times \R^n \to \R$ satisfying $$\int_{0}^T\biggl(\sum_{i=0}^k\int_{B} |\nabla^k f|^p\biggr)^{\frac{q}{p}} \, \d x \d t  < \infty,$$ for every ball $B \subset \R^n$. We denote $C([0,T];\mathcal{B})$
 as the set of functions $f: [0,T] \to \mathcal{B}$ which are continuous, and $C([0,T]; \L^p_{\loc}(\R^n))$ is made of measurable functions $f: [0,T] \to \L^p_{\loc}(\R^n)$ such that $f \in C([0,T]; \L^p(B))$ for every ball $B \subset \R^n.$
 
 We define $A_2(\R^n)$ as all the positive functions $w \in \L^1_{\loc}(\R^n)$ such that for a constant $C>0$ we have
\begin{align}
  \label{eq:A2cond}
    \frac{\int_{B} w \, \d x \int_{B} w^{-1} \, \d x}{|B|^2} \leq C,
\end{align}
 for every ball $B \subset \R^n$. The smallest constant $C$ in \eqref{eq:A2cond} is denoted by $[w]_{A_2}$. 
 Now, let $w, w^{-1} \in \L^1_{\loc}(\R^n).$ We define the measures $\d w(x):= w(x) \, \d x, \d w^{-1}(x) :=w^{-1}(x) \, \d x.$ For every Borel subset $B \subset \R^n$, define $w(B) := \int_{B} \, \d w, w^{-1}(B) := \int_{B} \, \d w^{-1}.$ The space $\L^2_w(\R^n)$ is defined by functions $f \in \L^1_{\loc}(\R^n)$ such that $$\|f\| ^2_{\L^2_w(\R^n)} := \int_{\R^n} f^2 \, \d w < \infty,$$
 and $\L^2_{w,\loc}(\R^n)$ includes all measurable $f:\R^n \to \R$ such that $\int_{B} f^2 \, \d w < \infty$ for every ball $B \subset \R^n.$ The Sobolev space $\H^1_w(\R^n)$ is defined by all the functions $f \in \L^2_w(\R^n)$ such that $\int_{\R^n} |\nabla f|^2 \, \d w< \infty$ and for a sequence $f_i \in C^{\infty}_0(\R^n)$, we have 
 \begin{align*}
    \lim_{i \to \infty} \int_{\R^n} |f_i-f|^2 \, \d w + \int_{\R^n} |\nabla (f_i-f)|^2 \, \d w=0.
 \end{align*}
 Note that by Cauchy-Schwarz inequality, we have 
 \begin{align*}
 \biggl(\int_{B} |\nabla f| \, \d x \biggr)^2 \leq w^{-1}(B)    \int_{B} |\nabla f|^2 \, \d w,
 \end{align*} for every $f \in \H^1_w(\R^n)$ and ball $B \subset \R^n$. Hence, $\H^1_w(\R^n) \subset W^{1,1}_{\loc}(\R^n)$. In the case that $w(x) = 1$, we denote $\H^1_w(\R^n)$ by $\H^1(\R^n)$.
The norm $\|\, \|_{\H^1_w(\R^n)}$ is defined by 
\begin{align*}
  \|f\|^2_{\H^1_w(\R^n)} :=  \int_{\R^n} f^2 \, \d w + \int_{\R^n} |\nabla f|^2 \, \d w,
\end{align*}
for every $f \in \L^1_{\loc}(\R^n).$ The space $W^{1,2}_w(\R^n)$ includes functions $f \in \L^1_{\loc}(\R^n)$ such that $\|f\|_{\H^1_w(\R^n)}<\infty.$ 
Note that, unlike the classical Sobolev space, it is possible for special weights $\H^1_w(\R^n) \neq W^{1,2}_w(\R^n),$ see \cite{CS, Z}. We say that $w$ is a regular weight if $\H^1_w(\R^n) = W^{1,2}_w(\R^n).$ Let $U \subset \R^m$ be a Lebesgue measurable subset for a positive integer $m$. We denote the set of all the bounded and continuous real-valued functions on $U$ by $C_b(U)$. The space $C_0(U)$ includes all the functions in $C_b(U)$ with compact support. We represent the space of all bounded signed Borel measures on $U$ by $\mathcal{M} (U)$. By the Jordan decomposition theorem, for every $\mu \in \mathcal{M} (U)$ there exists two unique bounded positive Borel measures $\mu^+, \mu^-$ such that $\mu = \mu^+ - \mu^-.$ We define the Borel measure $|\mu|:= \mu^+ + \mu^-$. If $\mu, \nu \in \mathcal{M}(U)$ and $\mu$ is absolutely continuous with respect to $\nu$, then we denote $\frac{d \mu}{\d \nu}: U \to \R$ as the bounded $\nu$-measurable function satisfying 
\begin{align*}
    \mu(F) = \int_{F } \frac{d \mu}{\d \nu} \, \d \nu,
\end{align*}
for every Borel subset $F \subset U.$

\subsection{Coefficient}
We assume that the matrix valued function $A = (a_{ij})_{i,j=1}^n$ is with real measurable entries $A_{ij}: [0,T] \times \R^n \to \R$ which, for some positive $\lambda,\Lambda$, satisfy the elliptic conditions 
\begin{align}
\label{eq:ellipticcond}
    |A(t,x) \chi \cdot \xi| \leq \Lambda |\chi| \, |\xi|, \quad A(t,x) \chi \cdot \chi \geq \lambda  |\chi|^2,
\end{align}
for all $\chi,\xi \in \R^n$ and a.e. $(t,x) \in [0,T] \times \R^n.$ The matrix valued function $A^*$ is the transpose of $A.$

\subsection{Weak and vague convergence of measures}
We apply two main convergences for finite signed Borel measures here. First, we say that $\mu_i \in \mathcal{M}([0,T] \times \R^n)$ is weakly convergent to $\mu \in \mathcal{M}([0,T] \times \R^n)$ if 
\begin{align*}
    \lim_{i \to \infty} \int_0^T \int_{\R^n} f \, \d \mu_i =\int_0^T \int_{\R^n} f \, \d \mu,
\end{align*}
for every $f \in C_b([0,T] \times \R^n).$ Second, we say that $\mu_i \in \mathcal{M}([0,T] \times \R^n)$ converges vaguely to $\mu \in \mathcal{M}([0,T] \times \R^n)$ if
\begin{align*}
    \lim_{i \to \infty} \int_0^T \int_{\R^n} f \, \d \mu_i =\int_0^T \int_{\R^n} f \, \d \mu,
\end{align*}
for every $f \in C_0^{\infty}([0,T] \times \R^n).$ Note that every weak convergence is a vague convergence, but the reverse may not hold.

\subsection{Truncation function}
For every measurable function $u: \R^n \to \R$ and $k>0$, define 
\begin{align*}
    T_k(u)(x) :=  \begin{cases}
        u(x) \quad  &|u(x)| \leq k,\\
        k \sgn(u(x)) \quad & |u(x)|> k.
    \end{cases}
\end{align*}
Then, by a standard result if $u \in W^{1,p}_{\loc}(\R^n)$, for a $p \geq 1$ , then $T_k(u) \in W^{1,p}_{\loc}(\R^n)$ and $\|\nabla T_k(u)\|_{\L^p(U)} \leq \|\nabla u\|_{\L^p(U)} $ for every $U \subset \R^n$. For $p>1$, the space $\mathcal{T}^p(\R^n)$ includes all the measurable functions $f: \R^n \to \R$ such that $T_k f \in W^{1,p}_{\loc}(\R^n)$ for every $k>0$.
We refer to \cite{BBGGPV} for the properties of the spaces $\mathcal{T}^p(\R^n).$

\section{Existence for integrable drifts}
In this section, we prove Theorem \ref{thm:boundeddrift}. We divide the proof into several steps.

Let $T>0, v \in \L^{\infty}([0,T]\times \R^n,\R^n)$. To define the notion of a solution, we need the notion of the dual solution.

\begin{defn} \label{def:dualsol}
Let $f \in \L^{\infty}([0,T] \times \R^n) \cap \L^{2}([0,T] \times \R^n)$. We say that $\phi$ is an adjoint convection solution of 
\begin{align}
\label{eq:dualsol}
    -\partial_t \phi - \div A^* \nabla \phi - v \cdot \nabla \phi = f,
\end{align}
in $[0,T] \times \R^n$ if $ \phi \in \L^2([0,T]; \H^1(\R^n)) \cap C([0,T]; \L^2(\R^n))$, $\phi(T,x)=0$ for a.e. $x \in \R^n$, and 
\begin{align*}
    &\int_{\R^n} \phi \, \psi\, \d x \,\bigg|_{0} + \int_0^T \int_{\R^n} \phi\, \partial_t \psi  \, \d x \d t + \int_0^T \int_{\R^n} A^* \nabla \phi \cdot  \nabla \psi  \, \d x \d t \\& -\int_0^T \int_{\R^n} \psi\,  v \cdot \nabla \phi  \, \d x \d t  = \int_{\R^n} f\, \phi \, \d x \d t,
\end{align*}
for every $\psi \in C^{\infty}_0([0,T] \times \R^n).$ 

\end{defn}
\begin{prop}
\label{prop:dualsol}
   For every $f \in \L^{\infty}([0,T]\times \R^n) \cap \L^2([0,T] \times \R^n)$, there exists a unique adjoint convection solution $\phi$ for \eqref{eq:dualsol} which satisfies
   \begin{align*}
   (T-t) \inf f \leq \phi(t,x) &\leq (T-t) \sup f ,\\
    \| \phi \|_{\L^{2}([0,T];\H^1 (\R^n))} &\leq  \frac{2}{\lambda} e^{(2+ \|v\|^2_{\L^{\infty}}/\lambda)T} \|f\|_{\L^{2}([0,T] \times \R^n)},
   \end{align*}
   for a.e. $(t,x) \in [0,T] \times \R^n.$ Moreover, for every pointwise approximation of $f,v,A$ by uniformly bounded sequences $f_i \in  \L^2([0,T] \times \R^n) \cap \L^{\infty}([0,T] \times \R^n), v_i \in \L^{\infty}([0,T] \times \R^n,\R^n),$ and $A_i \in \L^{\infty}([0,T] \times \R^n,\R^n \times \R^n)$, where $A_i$ satisfies \eqref{eq:ellipticcond} and $$\lim_{i \to \infty} \|f_i-f\|_{\L^2([0,T] \times \R^n)}=0,$$ there exists a sequence of functions $\phi_i \in \L^2([0,T]; \H^1(\R^n)) \cap C([0,T]; \L^2(\R^n))$ satisfying
    \begin{align*}
        -\partial_t \phi_i - \div A_i^* \nabla \phi_i &-v_i \cdot \nabla \phi_i=f_i,
    \end{align*}
in $[0,T] \times \R^n$ with $\phi_i(T,x)=0$ for a.e. $x \in \R^n$, where $\phi_i$ converges to $\phi$ in $C([0,T];\L^1_{\loc}(\R^n))$ and $\L^2([0,T];\H^1_{\loc}(\R^n))$.
    
\end{prop}
\begin{proof}
 To prove uniqueness, let $\phi_1,\phi_2 \in \L^2([0,T]; \H^1(\R^n)) \cap C([0,T];\L^2(\R^n))$ be two adjoint convection solutions of \eqref{eq:dualsol}. Then, $\phi :=\phi_1 - \phi_2 \in \L^2([0,T]; \H^1(\R^n)) \cap C([0,T];\L^2(\R^n))$ is an adjoint convection solution of \begin{align*}
     -\partial_t \phi - \div A^* \nabla \phi - v \cdot \nabla \phi = 0,
\end{align*} in $[0,T] \times \R^n$ with $\phi(T,x)=0$ for a.e. $x \in \R^n.$ Let $\varphi_{\epsilon}$ be a non-negative approximation of identity in $[0,T] \times \R^n$. Then,
$$-\partial_t \varphi_{\epsilon} \ast \phi - \div \varphi_{\epsilon} \ast (A^{\ast} \nabla \phi) - \varphi_{\epsilon} \ast (v \cdot \nabla \phi) =0,$$ weakly in $[\epsilon,T-\epsilon] \times \R^n$. Hence, by taking $\varphi_{\epsilon} \ast \phi$ as a test function and using Young's inequality, we obtain
\begin{align*}
   &- \frac{1}{2} \partial_t \int_{\R^n} |\varphi_{\epsilon} \ast \phi|^2 \, \d x\biggl |_t \, \d x + \int_{\R^n}  \varphi_{\epsilon} \ast (A^* \nabla \phi)  \cdot \varphi_{\epsilon} \ast \nabla  \phi \,  \d x \bigg |_t \\& \leq \frac{\lambda}{2} \int_{\R^n} \varphi_{\epsilon} \ast |\nabla \phi|^2 \, \d x \bigg |_t + \frac{ \|v\|^2_{\L^{\infty}}}{2\lambda} \int_{\R^n} |\varphi_{\epsilon} \ast \phi|^2 \, \d x  \bigg |_t
\end{align*}
for every $t \in [\epsilon,T-\epsilon].$ Now, by Gr\"{o}nwall's inequality, we derive that 
\begin{align*}
  &\frac{1}{2}  e^{\biggl(\frac{ \|v\|^2_{\L^{\infty}}}{\lambda} \epsilon \biggr)} \int_{\R^n} |\varphi_{\epsilon} \ast \phi|^2 \, \d x \biggl |_{\epsilon}+  \int_{\epsilon}^{T-\epsilon} \int_{\R^n} e^{\biggl(\frac{ \|v\|^2_{\L^{\infty}}}{\lambda} t\biggr)}  (\varphi_{\epsilon} \ast A^* \nabla \phi)  \cdot \varphi_{\epsilon} \ast \nabla  \phi \, \d t  \d x \\& \leq  \frac{\lambda}{2} \int_{\epsilon}^{T-\epsilon} \int_{\R^n} e^{\biggl(\frac{ \|v\| ^2_{\L^{\infty}}}{\lambda} t\biggr)}  \varphi_{\epsilon} \ast |\nabla \phi|^2 \, \d x \d t + \frac{1}{2}  e^{\biggl(\frac{ \|v\|^2_{\L^{\infty}}}{\lambda} (T-\epsilon) \biggr)} \int_{\R^n} |\varphi_{\epsilon} \ast \phi|^2 \, \d x \biggl |_{T-\epsilon}.
\end{align*}
By taking $\epsilon \to 0$ in the above inequality and using the elliptic condition \eqref{eq:ellipticcond} as well as $\phi \in \L^2([0,T];\H^1(\R^n)) \cap C([0,T];\L^2(\R^n))$, it is implied that
\begin{align*}
    \int_{0}^T \int_{\R^n} e^{\biggl(\frac{ \|v\|^2_{\L^{\infty}}}{\lambda} t\biggr)}  |\nabla \phi|^2 \, \d x \d t  =0.
\end{align*}
 In conclusion, $\phi$ only depends on time and $- \partial_t \phi =0$ weakly in $[0,T] \times \R^n.$ This together with $\phi_i(T,x)=0$ a.e. in $\R^n$ implies that $\phi=0$ a.e. in $[0,T] \times \R^n.$ For the existence part, we follow the argument in \cite[Lem. 3.4]{P2}. Let  $f_i \in  C^{\infty}_0([0,T] \times \R^n), v_i \in C^{\infty}([0,T] \times \R^n, \R^n)$, and $A_i \in C^{\infty}([0,T] \times \R^n,\R^n \times \R^n)$ be uniformly bounded sequence of functions converging to $f,v,A$ pointwise, where $A_i$ satisfies \eqref{eq:ellipticcond} and $\lim_{i \to \infty} \|f_i-f\|_{\L^2([0,T] \times \R^n)}=0.$ The construction of such sequences is derived by a standard mollification argument. By a standard result on parabolic equations, see \cite{A}, there exist smooth functions $\phi_i \in \L^2([0,T];\H^1( \R^n))$ satisfying $$ -\partial_t \phi_i - \div A_i^* \nabla \phi_i -v_i \cdot \nabla \phi_i=f_i,$$ weakly in $[0,T] \times \R^n$ with $\phi(T,x)=0$ for every $x \in \R^n.$ Hence,
\begin{align*}
    &-\partial_t (\phi_i-(T-t)\sup f_i ) - \div A_i^* \nabla (\phi_i-(T-t)\sup f_i )\\& -v_i \cdot \nabla (\phi_i-(T-t)\sup f_i ) \leq 0,
\end{align*}
 weakly in $[0,T] \times \R^n$. By the maximum principle, see \cite{A},
 $  \phi_i-(T-t)\sup f_i \leq 0.$
Likewise, one can prove that $\phi_i-(T-t)\inf f_i \geq 0$. Now, by taking $\phi_i$ as a test function, we derive that
\begin{align*}
-\frac{1}{2} \partial_t \|\phi_i(t,\cdot)\|^2_{\L^2(\R^n)} + \lambda\|\nabla \phi_i(t,\cdot)\|^2_{\L^2(\R^n)}& \leq \int_{\R^n} \phi_i \, v_i \cdot \nabla \phi_i \, \d x \\& +  \int_{\R^n} f_i\, \phi_i \, \d x .
\end{align*}
for every $0\leq t \leq T$. Hence, by Young and Gr\"{o}nwall's inequalities, we derive that
\begin{align}
\label{eq:energyboundapprox}
    \sup_{0\leq t \leq T}\frac{1}{2} \|\phi_i(t,\cdot)\|^2_{\L^2(\R^n)} + \frac{\lambda}{2}\|\nabla \phi_i(t,\cdot)\|^2_{\L^2([0,T] \times \R^n)}& \leq e^{  \biggl(1+\frac{\|v_i\|^2_{\L^{\infty}}}{\lambda}\biggr)T}\|f_i\|^2_{\L^2([0,T] \times \R^n)}.
\end{align}
Then, by using the equation, it is obtained that $\phi_i$ is bounded uniformly in $\L^2([0,T];\H^1(\R^n))$ and $\partial_t \phi_i$ is bounded uniformly in $\L^2([0,T];\H^{-1}(\R^n)) +\L^1([0,T]; \L^1_{\loc}(\R^n)).$ Hence, by the classical result in \cite{Simon}, $\phi_i$ converges to $\phi \in \L^2([0,T]; \H^1(\R^n))$ pointwise a.e. in $[0,T] \times \R^n$ and strongly in $\L^2([0,T];\L^2_{\loc}(\R^n))$, up to a subsequence. Moreover, $\phi$ satisfies \eqref{eq:dualsol} weakly in $[0,T] \times \R^n$ and
\begin{align}
\label{eq:linftybound}
   (T-t) \inf f \leq \phi(t,x) &\leq (T-t) \sup f.
\end{align}
Since $\phi \in \L^2([0,T];\H^1(\R^n))$ and $\partial_t \phi \in \L^2([0,T];\H^{-1}(\R^n)) +\L^1([0,T]; \L^1_{\loc}(\R^n))$, by \cite[Th. 1]{P1}, we deduce that $\phi \in C([0,T]; \L^1_{\loc}(\R^n))$. Then, by \eqref{eq:linftybound}, we get
\begin{align*}
    \int_{B} |\phi(t,x) - \phi(s,x)|^2 \, \d x \leq 2 T \|f\|_{\L^{\infty}} \int_{B} |\phi(t,x) - \phi(s,x)| \, \d x,
\end{align*}
for every $0 \leq s\leq t \leq T$
and ball $B \subset \R^n.$ Hence, $\phi \in C([0,T]; \L^2_{\loc}(\R^n)).$ 

Now, we demonstrate that $\phi_i$ converges to $\phi$ in the space $C([0,T];\L^1_{\loc}(\R^n))$. Let us fix $\varphi \in C^{\infty}_0([0,T] \times \R^n).$ By taking the differences of the equations for $\phi, \phi_i$ and using $(\phi-\phi_i) \, \varphi^2$ as a test function, we obtain
\begin{align*}
      &\lambda \int_0^T \int_{\R^n} |\nabla  (\phi-\phi_i)|^2 \, \varphi^2 \, \d x \d t \\&\leq  \int_0^T \int_{\R^n} (A_i^*-A^*)
\nabla \phi \cdot (\varphi^2\nabla  (\phi-\phi_i) + 2(\phi-\phi_i)\varphi \nabla \varphi) \,  \d x \d t  
   \\& -\int_0^T\int_{\R^n} \varphi^2 \,(\phi_i-\phi) \,  v_i \cdot \nabla (\phi_i-  \phi) \, \d x \d t
     -\int_0^T\int_{\R^n} \varphi^2 \, (\phi_i-\phi)\, (v_i -v) \cdot \nabla  \phi   \, \d x \d t
\\&  + \int_0^T \int_{\R^n} \varphi^2 \, (f_i- f) \, (\phi-\phi_i) \, \d x \d t
\end{align*}
for $t \in [0,T].$ In conclusion, by using strong convergence of $\phi_i$ to $\phi$ in $\L^2([0,T];\L^2_{\loc}(\R^n))$, \eqref{eq:energyboundapprox}, and Lebesgue's dominated convergence, we obtain $\phi_i$ converges to $\phi$ strongly in the space $\L^2([0,T]; \H^1_{\loc}(\R^n)).$ Then, by using the equations for $\phi, \phi_i$, we derive that $\partial_t \phi_i$ converges to  $\partial_t \phi$ in $\L^2([0,T]; \L^2(B))+\L^2( [0,T];\H^{-1}(B))$ for every ball $B \subset \R^n$. Hence, by \cite[Thm. 1]{P1}, $\phi_i$ converges to $\phi$ in $C([0,T];\L^1_{\loc}(\R^n))$. Note that the argument above works also for any pointwise approximation of $f,v,A$ by uniformly bounded sequences $f_i \in  \L^2([0,T] \times \R^n) \cap \L^{\infty}([0,T] \times \R^n), v_i \in \L^{\infty}([0,T];\L^{\infty}(\R^n \times \R^n)), A_i \in \L^{\infty}([0,T] \times \R^n,\R^n \times \R^n)$, where $A_i$ satisfies \eqref{eq:ellipticcond} and $\lim_{i \to \infty} \|f_i-f\|_{\L^2([0,T] \times \R^n)}=0$. Finally, we prove that in fact $\phi \in C([0,T];\L^2(\R^n)).$ 
To elaborate the proof, let $r>0$ and $\psi \in C^{\infty}(\R^n)$ be a function which satisfies $0 \leq \psi \leq 1$ in $\R^n$, $\psi = 0$ in $B(0,r/2),$ $\psi = 1$ in $\R^n \setminus B(0,r),$ and $|\nabla \psi| \leq \frac{C}{r}$ in $\R^n$ for a constant $C>0.$ By using the equation for $\phi_i$, we take $\psi \, \phi_i$ as a test function to arrive at
\begin{align*}
  & - \frac{1}{2}\partial_t \int_{\R^n} \phi_i^2 \, \psi^2 \, \d x \bigg |_t + \int_{\R^n} A_i^* \nabla \phi_i \cdot  \nabla (\phi_i\, \psi^2) \, \d x \bigg |_t  \leq \int_{\R^n} \|v_i\|_{\L^{\infty}} |\nabla \phi_i| \, |\phi_i| \, \psi ^2 \, \d x \bigg |_t \\& + \int_{\R^n} |f_i| \, |\phi_i| \, \psi^2 \, \d x \bigg |_t.
\end{align*}
Then, by using Young's inequality and Gr\"{o}nwall's inequality, we obtain
\begin{align*}
   &\sup_{0 \leq t \leq T} \frac{1}{2}\int_{\R^n} \phi_i^2 \, \psi^2 \, \d x \bigg |_t 
 \\&\leq e^{\biggl(1+\frac{\|v_i\|^2_{\L^{\infty}}}{\lambda}\biggr)T} \biggl(\int_0^T \int_{\R^n} \frac{2 \Lambda^2}{\lambda} \phi_i^2  |\nabla \psi|^2  \, \d x \d t +\int_0^T \int_{\R^n} f_i^2 \, \psi^2 \, \d x \d t \biggr). 
\end{align*}
Hence, by strong convergence of $\phi_i$ to $\phi$ in both $C([0,T];\L^1_{\loc}(\R^n))$ and $\L^2([0,T];\L^2_{\loc}(\R^n))$, Fatou's lemma, and Lebesgue's dominated convergence, we obtain
\begin{align*}
  &  \sup_{0 \leq t \leq T} \frac{1}{2}\int_{\R^n} \phi^2 \, \psi^2 \, \d x \bigg |_t 
\\&\leq e^{\biggl(1+\frac{\|v\|^2_{\L^{\infty}}}{\lambda}\biggr)T} \biggl( \frac{2 (C\, \Lambda)^2}{(r\, \lambda)^2}  \int_0^T \int_{B(0,r)} \phi^2   \, \d x \d t +\int_0^T \int_{\R^n \setminus B(0,r/2)} f^2  \, \d x \d t \biggr).
\end{align*}
Thus, we obtain $$\lim_{r \to 0} \sup_{0 \leq t \leq T} \int_{\R^n \setminus B(0,r)} \phi^2 \, \d x \bigg |_t =0.$$ Since $\phi \in C([0,T]; \L^2_{\loc}(\R^n))$, for every $t \in [0,T]$, $\epsilon>0,$ ball $B \subset \R^n$, there exists a constant $\delta(t,\epsilon,B)>0$ depending on $t, \epsilon,B$, such that 
\begin{align*}
    \int_{B}|\phi(t,x)-\phi(s,x)|^2 \, \d x < \epsilon,
\end{align*}
for all $s \in  [0,T]$ satisfying $|s-t|< \delta(t,\epsilon,B).$ Let $\epsilon>0$ and $r>0$ satisfies 
\begin{align*}
    \sup_{0 \leq t \leq T}\int_{\R^n \setminus B(0,r)} \phi^2 \, \d x \bigg |_t < \frac{\epsilon}{8}.
\end{align*}
Then, for every $t>0$, we have
\begin{align*}
&\int_{\R^n}|\phi(t,x)-\phi(s,x)|^2 \, \d x = \int_{B(0,r)}|\phi(t,x)-\phi(s,x)|^2 \, \d x  + \int_{\R^n \setminus B(0,r)}|\phi(t,x)-\phi(s,x)|^2 \, \d x \\& < \frac{\epsilon}{2} + 2 \int_{\R^n \setminus B(0,r)} \phi^2 \, \d x \bigg |_t +  2 \int_{\R^n \setminus B(0,r)} \phi^2 \, \d x \bigg |_s  < \epsilon,
\end{align*}
for every $s \in [0,T]$ satisfying $|s-t| \leq \delta(t,\epsilon/2,B(0,r)).$ In conclusion, $\phi \in C([0,T]; \L^2(\R^n))$. This completes the proof.

\end{proof}
 Now, we define a notion of solution for convection-diffusion equations.
\begin{defn}
\label{def:convec1}
    Let $T>0, g \in \L^1(\R^n)$, and $R \in \L^1([0,T]\times \R^n).$ Then, we say that $u$ is a convection solution for 
\begin{align*}
    \partial_t u -  \div A \nabla u + \div (v\, u) = R,
\end{align*}
in $[0,T] \times \R^n$ with initial value $g$ if $u \in \L^1([0,T]\times \R^n)$ and
\begin{align*}
     \int_0^T \int_{\R^n} u\, f\, \d x \d t =\int_{\R^n} g \, \phi \, \d x \bigg |_0 + \int_{0}^T \int_{\R^n} R\, \phi \, \d x \d t,
\end{align*}
    for every $f \in \L^{\infty}([0,T] \times \R^n) \cap \L^2([0,T] \times \R^n)$ and dual solution $\phi$ satisfying \eqref{eq:dualsol}. We say that a convection solution $u$ for
    \begin{align*}
        \partial_t u - \div A \nabla u + \div(v \, u) = R,
    \end{align*}
  in $[0,T] \times \R^n$ with the initial value $g$ satisfies the conservation of mass if 
\begin{align*}
    \int_{\R^n} u \, \d x \bigg |_s = \int_{\R^n} g \, \d x + \int_0^s \int_{\R^n} R \, \d x \d t,
\end{align*}
for a.e. $s \in [0,T].$
    
\end{defn}

Now, we prove Theorem \ref{thm:boundeddrift}.

\begin{proof}[Proof of Theorem \ref{thm:boundeddrift}]
The proof is divided into several steps.\\
\textbf{1}. First, we prove the uniqueness of convection solutions. If $u_1, u_2 \in \L^1([0,T] \times 
\R^n)$ are convection solutions of \eqref{eq:dualsol} with the initial value $g,$ then
\begin{align*}
   \int_0^T \int_{\R^n} (u_1-u_2) f \, \d x \d t =0
\end{align*}
  for every $f \in C^{\infty}_0([0,T] \times \R^n).$ Hence, $u_1=u_2$ a.e. in $[0,T] \times \R^n.$\\
 \textbf{2}. In this step, we prove some properties of approximate solutions. For now, we assume that $R,g$ are non-negative and $(\ln{g})^- \in \L^1_{\loc}(\R^n)$. We remove the extra assumptions later in Step 6. Let $\varphi_{\epsilon}$ be a positive approximation of identity. We define $ \tilde R_i =\varphi_{1/i} \ast R, \tilde g_i := \varphi_{1/i} \ast g$, $ A_i= \varphi_{1/i} \ast A$, and $v_i := \varphi_{1/i} \ast v$.  Then, $R_i = \tilde R_i + \frac{1}{i} e^{-x^2}, g_i = \Tilde{g}_i + \frac{1}{i} e^{-x^2}$ are positive smooth approximations of $R,g$ in $\L^1([0,T];\L^1(\R^n))$, respectively. Note that, by Young's convolution inequality, $R_i \in \L^p([0,T] \times \R^n),g_i \in \L^p(\R^n)$ for every $p \geq 1$ and $i.$ Hence, by a standard theorem on parabolic equations, see \cite{A}, there exists a smooth function $u_i \in \L^2([0,T]; \H^1(\R^n))$ such that $u_i(0,x) = g_i(x)$ for every $x \in \R^n$ and 
  \begin{align}
  \label{eq:approxsol}
\partial_t u_i - \div A_i \nabla u_i + \div (v_i\, u_i ) = R_i,     
  \end{align}
 in $[0,T] \times \R^n.$ By the strong maximum principle, see \cite{A}, $u_i >0$ in $[0,T] \times \R^n.$ Then, taking $u_i^{\alpha}$ as a test function for $\alpha >0$, we obtain 
\begin{align*}
   &\frac{1}{\alpha + 1} \partial_t  \int_{\R^n} u_i^{\alpha+1} \, \d x \bigg |_t  + \alpha \int_{\R^n} u_i^{\alpha-1} \, A_i \nabla u_i \cdot \nabla u_i \, \d x \bigg |_t \\& - \alpha \int_{\R^n} u_i^{\alpha} \, v_i \cdot \nabla u_i \, \d x \bigg |_t = \int_{\R^n} R_i u_i^{\alpha} \, \d x \bigg |_t,
\end{align*}
for every $t \in [0,T]$. Then, by Young's inequality, we obtain
\begin{align*}
     &\frac{1}{\alpha + 1} \partial_t  \int_{\R^n} u_i^{\alpha+1} \, \d x \bigg |_t  +  \frac{\alpha \lambda}{2} \int_{\R^n} |\nabla u_i|^2 u_i^{\alpha-1} \, \d x \bigg |_t \\& \leq  \biggl( \frac{ \alpha\|v_i\|^2_{\L^{\infty}}}{2\lambda } + \frac{\alpha}{\alpha+1}\biggr) \int_{\R^n} u_i^{\alpha+1} \, \d x+ \frac{1}{\alpha}  \int_{\R^n} R_i^{\alpha} \, \d x \bigg |_t.
\end{align*}
for every $t \in [0,T].$ Hence, by Gr\"{o}nwall's inequality, we have
\begin{align}
\label{eq:integestim}
   \sup_{0 \leq t \leq T} \int_{\R^n} u_i^{\alpha+1} \, \d x \d t \leq e^{\big ( \frac{(\alpha+1) \|v_i\|^2_{\L^{\infty}} }{2 \lambda} +\alpha \big ) T}\biggl(\frac{\alpha +1}{\alpha}\int_{0}^T \int_{\R^n} R_i^{\alpha} \, \d x \d t + \int_{\R^n} g_i^{\alpha +1} \, \d x\biggr).
\end{align}
 Now, take $s \in [0,T]$ and a sequence  functions $\psi_j \in C_0^{\infty}(\R^n)$ satisfying
 \begin{equation}
 \label{eq:spectestfunc}
 \begin{aligned}
   1_{B(0,j)} &\leq  \psi_j \leq 1_{B(0,2j)},\\
   \psi_j &\leq \psi_{j+1},\\
   |\nabla^k \psi_j| &\leq \frac{C_k}{j^k},
 \end{aligned}
 \end{equation}
 for all non-negative integers $j,k$, where $C_k$ is a constant depending on $n,k.$ Then,
\begin{equation}
\begin{aligned}
\label{eq:masequ}
    \int_{\R^n} u_i \, \psi_j \, \d x\bigg|_s&= \int_{\R^n} g_i\, \psi_j \, \d x \bigg |_0 + \int_0^s \int_{\R^n} u_i \div A_i \nabla \psi_j \, \d x \d t \\& + \int_0^s \int_{\R^n} u_i \, v_i \cdot \nabla \psi_j \, \d x \d t + \int_{0}^s\int_{\R^n} R_i \, \psi_j \, \d x \d t,
\end{aligned}
\end{equation}
for every $0\leq s \leq T$. Hence, using \eqref{eq:integestim} and Hölder's inequality, we deduce
\begin{align*}
&  \biggl|  \int_0^s \int_{\R^n} u_i \div A_i \nabla \psi_j \, \d x \d t- \int_0^s \int_{\R^n} u_i v_i \cdot \nabla \psi_j \, \d x \d t \biggr| \leq \frac{C_1+ C_2+1}{j} \int_0^s \int_{B(0,2j)} u_i \, \d x \d t \\&\leq (|B(0,1)|(2j)^n)^{\frac{\alpha}{\alpha+1}} \frac{C_1+ C_2+1}{j} \biggl(\int_{0}^s\int_{B(0,2j)} u_i^{\alpha+1} \, \d x \d t\biggr)^{\frac{1}{\alpha+1}}
\\& \leq C (2j)^{\frac{n \alpha}{\alpha+1}-1}T^{\frac{1}{\alpha+1}} e^{\big ( \frac{ (\alpha+1) \|v_i\|^2_{\L^{\infty}} }{2\lambda} +\alpha \big ) T}\biggl(\frac{\alpha +1}{\alpha}\int_{0}^T \int_{\R^n} R_i^{\alpha} \, \d x \d t + \int_{\R^n} g_i^{\alpha +1} \, \d x\biggr),
\end{align*}
for $\alpha >0$, where $C$ is a constant depending on $n$. Taking $\alpha$ small enough ( $\alpha < \frac{1}{n-1}$), letting $j \to \infty$ in \eqref{eq:masequ}, and using monotone convergence theorem, we arrive at    \begin{align}
    \label{eq:masconserv}
        \int_{\R^n} u_i \, \d x\bigg|_s = \int_{\R^n} g_i \, \d x  + \int_0^s \int_{\R^n} R_i \, \d x \d t.
    \end{align}
    \textbf{3}. In this step, we prove logarithm-type estimates.
Since $u_i >0$, we can define the smooth function $S_i := \ln u_i$. Hence, by \eqref{eq:approxsol}, it is implied that 
\begin{align}
\label{eq:entropyequ}
    \partial_t S_i - \div A_i \nabla S_i - A_i \nabla S_i \cdot \nabla S_i + \div v_i + v_i \cdot \nabla S_i = R_i e^{-S_i},
\end{align}
weakly in $[0,T] \times \R^n$. Note that by \eqref{eq:masconserv}, we obtain
\begin{align}
\label{eq:estpospart}
 \int_{\R^n} S_i^+ \, \d x \bigg|_s \leq \int_{\R^n} e^{S_i} \, \d x \bigg|_s =\int_{\R^n} g_i \, \d x  + \int_0^s \int_{\R^n} R_i \d x \d t.
\end{align}
for every $s \in [0,T].$ We need to prove estimates for $S_i^-.$ By using $\varphi^2$, where $\varphi \in C^{\infty}_0(\R^n)$, as a test function and \eqref{eq:entropyequ}, we get
\begin{equation}
\begin{aligned}
   &\int_0^s \int_{\R^n} \varphi^2 \, A \nabla S_i \cdot \nabla S_i \, \d x \d t + \int_0^s \int_{\R^n} R_i e^{-S_i} \varphi^2 + \int_{\R^n} S_i \varphi^2 \bigg |_0 = \int_{\R^n} S_i \varphi^2 \bigg |_s \\&+\int_0^s \int_{\R^n}2 \varphi\,  A \nabla S_i \cdot \nabla \varphi \, \d x \d t - \int_0^s \int_{\R^n}2 \varphi \, v_i \cdot \nabla \varphi \, \d x \d t + \int_{0}^s \int_{\R^n} \varphi^2 \, v_i \cdot \nabla S_i \, \d x \d t,
\end{aligned}
\end{equation}
for every $s \in [0,T].$ Then, by using Young's inequality again and \eqref{eq:estpospart}, we obtain
\begin{equation}
\label{eq:approxentropy}
\begin{aligned}
   &\frac{\lambda}{2}\int_0^T \int_{\R^n} |\nabla S_i|^2 \varphi^2 \, \d x \d t+  \int_0^T \int_{\R^n} R_i e^{-S_i} \varphi^2  \, \d x \d t + \sup_{0\leq s \leq T} \int_{\R^n} S_i^- \varphi^2 \, \d x \bigg |_s \\& \leq  2\int_{\R^n} S_i^- \varphi^2 \, \d x \bigg |_0   + 2\int_{\R^n} g_i \, \d x  + 2\int_0^T \int_{\R^n} R_i \, \d x \d t \\& + 6 \int_{0}^T \int_{\R^n} v_i^2 \varphi^2 \, \d x \d t +  \frac{4\Lambda^2}{\lambda} \int_{0}^T \int_{\R^n} |\nabla \varphi|^2 \, \d x \d t.
\end{aligned}
\end{equation}
\textbf{4}. Now, we use the estimates to get to a limit. By  \eqref{eq:entropyequ} and \eqref{eq:approxentropy}, for every ball $B \subset \R^n$, $\|S_i\|_{\L^2([0,T];\H^1(B))}, \|R_i e^{-S_i}\|_{\L^1([0,T] \times B)}$ is uniformly bounded and $\partial_t S_i$ has uniformly bounded norm in $\L^2([0,T];\H^{-1}(B)) + \L^1([0,T];\L^1(B)).$ Hence, by \cite[Thm. 1.1]{P1}, $S_i$ converges to $S \in \L^2([0,T]; \H^1_{\loc}(\R^n)) $ strongly in $\L^2([0,T]; \L^2_{\loc}(\R^n))$ and weakly in $\L^2([0,T];\H^1_{\loc}(\R^n))$, up to a subsequence. In particular, up to a subsequence, $S_i$ converges pointwise to $S$ a.e. in $[0,T] \times \R^n$. In conclusion, by Fatou's lemma and \eqref{eq:masconserv}, we derive $$\int_{0}^T\int_{\R^n} e^S \, \d x \d t \leq \liminf_{i \to \infty} \int_{0}^T \int_{\R^n} e^{S_i} \, \d x \d t = T\int_{\R^n} g \, \d x + T \int_{0}^T \int_{\R^n} R \, \d x \d t.$$
Also, by letting $i \to \infty$ in \eqref{eq:entropyequ} and using \cite[Thm. 4.3]{BM}, together with Fatou's lemma, we conclude \eqref{eq:strongmaxprin}. 
Now, we prove that $u:= e^S$ is the desired convection solution. Let us fix a function $f \in \L^{\infty}([0,T] \times \R^n) \cap \L^2([0,T] \times \R^n)$ and assume that $f_i \in C^{\infty}_0([0,T] \times \R^n)$ converges to $f$ in $\L^{\infty}([0,T] \times \R^n)$-weak* and strongly in $\L^2([0,T] \times \R^n)$. We define the sequence $f^k_i := f_i\, 1_{\{u_i \leq k\}}$. Since $u_i,v_i$ are smooth and satisfy \eqref{eq:approxsol} weakly, we obtain 
\begin{align}
\label{eq:convappro}
      \int_0^T\int_{\R^n} u_i \, f_i^k \, \d x \d t=  \int_{\R^n} g_i \, \phi^k_i \, \d x \bigg |_0 + \int_0^T \int_{\R^n} R_i \, \phi_i^k \, \d x \d t,  
\end{align}
for every dual solution $\phi_i^k$ satisfying 
\begin{align*}
    - \partial_t \phi_i^k - \div A_i^* \nabla \phi_i^k - v_i \cdot \nabla \phi_i^k = f_i^k,
\end{align*}
weakly in $[0,T] \times \R^n$. By Proposition \ref{prop:dualsol}, $\phi_i^k$ are uniformly bounded and $\phi^k, \phi^k(0,\cdot)$ converge pointwise to $\phi^k,\phi^k(0,\cdot)$, respectively. Moreover, $\phi^k$ satisfies 
\begin{align}
\label{eq:dualappro}
    - \partial_t \phi^k - \div A^* \nabla \phi^k - v \cdot \nabla \phi^k = f^k,
\end{align}
weakly in $[0,T] \times \R^n$ with $\phi^k(T,x)=0$ for a.e. $x \in \R^n$. Since $|u_i f_i^k| \leq k |f|$ and, up to a subsequence, $u_i$ converges pointwise to $u$, by Lebesgue's dominated convergence and taking $i \to \infty$ in \eqref{eq:convappro}, it is implied that 
\begin{align*}
     \int_{0}^T \int_{\R^n} u \, f\, 1_{\{u \leq k\}} \, \d x \d t =\int_{\R^n} g \, \phi^k \, \d x  \bigg |_0 +  \int_{0}^T \int_{\R^n} R \, \phi^k \, \d x \d t.
\end{align*}
By the same argument as above, up to a subsequence, $\phi^k$ and $\phi^k(0,\cdot)$ are uniformly bounded and converge pointwise to $\phi$ and $\phi^k(0,\cdot)$, respectively, where $\phi$ satisfies \eqref{eq:dualsol}. Hence, by the dominated convergence, we obtain
\begin{align*}
       \int_{0}^T \int_{\R^n} u \, f \, \d x \d t  = \int_{\R^n} g \, \phi \, \d x \bigg |_0 + \int_{0}^T \int_{\R^n} R \, \phi \, \d x \d t.
\end{align*}
    Thus, $u$ is the desired convection solution. Since $u_i \in \L^2([0,T]; \H^1(\R^n)) $ is the smooth solution of \eqref{eq:approxsol} with $u_i(0,x)=g_i(x)$ for $x \in \R^n$, by using the test function $T_k u_i$ for the equation of $u_i$, we obtain that 
    \begin{align*}
     &\sup_{0 \leq t \leq T}\int_{\R^n} \frac{1}{k}|T_k(u_i)|^2 \, \d x \bigg |_t +  \frac{\lambda}{2k} \int_0^T \int_{\R^n} |\nabla T_k u_i|^2 \, \d x \d t \\&\leq e^{\big(1+\frac{\|v_i\|^2_{\L^{\infty}}}{\lambda}\big)T} \biggl(\int_{\R^n} g_i \,  \d x + \int_{0}^T \int_{\R^n} R_i \, \d x \d t\biggr). 
     \end{align*}
    Hence, by Fatou's lemma and weak convergence of $\nabla T_k u_i$ to $\nabla T_k u$ in $\L^2([0,T] \times \R^n)$, it is implied that $u \in \mathcal{T}^2([0,T] \times \R^n)$ and
\begin{align*}
     &\sup_{0 \leq t \leq T}\frac{1}{k} \int_{\R^n} |T_k(u)|^2 \, \d x+ \frac{\lambda}{2k}\int_0^T \int_{\R^n} |\nabla T_k u|^2 \, \d x \d t \\&\leq e^{\big(1+\frac{\|v\|^2_{\L^{\infty}}}{\lambda}\big)T} \biggl(\int_{\R^n} g \,  \d x + \int_{0}^T \int_{\R^n} R \, \d x \d t\biggr).
\end{align*}
    By using \cite[Lem. 4.2]{BBGGPV}, we arrive at $u \in \L^2([0,T];W^{1,p}_{\loc}(\R^n))$ for all $ p < \frac{n}{n-1}$ and $u_i$ converges to $u$ weakly in $\L^2([0,T];\L^p_{\loc}(\R^n))$, up to subsequence depending on $p$. Moreover, 
    \begin{align*}
        \int_{0}^T \int_{\R^n} u_i \, f  \, \d x \d t= \int_{\R^n} g_i \, \phi \, \d x \bigg |_0 + \int_{0}^T \int_{\R^n} R_i \, \phi \, \d x \d t,
    \end{align*}
    for every $f \in C^{\infty}_0([0,T] \times \R^n)$ and $i.$ Hence, 
      \begin{align*}
         \lim_{i \to \infty}  \int_{0}^T \int_{\R^n} u_i \, f \, \d x \d t  = \int_{\R^n} g \, \phi \, \d x \bigg |_0+ \int_{0}^T \int_{\R^n} R\, \phi\, \d x \d t = \int_{0}^T \int_{\R^n} u \, f \, \d x \d t.
      \end{align*}
      In conclusion, $u_i$ converges to $u$ vaguely in $[0,T] \times \R^n$ without any need to pass to a subsequence.\\
\textbf{5}. In this step, we prove that there exists a set $F \subset [0,T]$ of measure zero such that
\begin{align}
\label{eq:conticonvinitial}
    \lim_{s \to 0;\, s \in [0,T] \setminus F} \int_{\R^n} u\, \psi \, \d x \bigg |_s= \int_{\R^n} g \, \psi \, \d x
\end{align}
for every $\psi \in C_b(\R^n).$ By the Stone–Weierstrass theorem and $\sigma$-compactness of $\R^n$, there exists a countable dense subset $D \subset C^{\infty}_0(\R^n).$  Let us fix an element $\varphi \in D$ and $0 \leq t_1 < t_2 \leq T$. Then, by using the equation of $u_i$ and test function $\varphi$, we have
\begin{align*}
 \biggl | \int_{\R^n} u_i \, \varphi \, \d x \bigg |_{t_1} -\int_{\R^n} u_i \, \varphi \, \d x \bigg |_{t_2}  \,\biggr | &\leq \int_{t_1}^{t_2} \int_{\R^n} \Lambda |\nabla  u_i| \, |\nabla \varphi| \, \d x \d t \\&+ \int_{t_1}^{t_2} \int_{\R^n} |v_i|\, |u_i| \,  |\nabla \varphi| \, \d x \d t + \int_{t_1}^{t_2} \int_{\R^n} R_i \, \varphi \, \d x \d t.
\end{align*}
Hence, by Hölder's inequality, we derive that
\begin{align*}
    & \limsup_{i \to \infty} \biggl | \int_{\R^n} u_i \, \varphi \, \d x \bigg |_{t_1} -\int_{\R^n} u_i \, \varphi \, \d x \bigg |_{t_2}  \,\biggr | \leq\\&\sqrt{t_2-t_1} |\supp \varphi|^{\frac{1}{p'}} \|\nabla \varphi\|_{\L^{\infty}(\R^n)}(\Lambda  + \|v\|_{\L^{\infty}})  \limsup_{i \to \infty} \|u_i+ |\nabla u_i|\|_{\L^2([0,T];\L^p(\supp \varphi))}\\& +\int_{t_1}^{t_2} \int_{\R^n} R \, \varphi \, \d x \d t,
\end{align*}
for every $p < \frac{n}{n-1}.$ Moreover, by \eqref{eq:masconserv}, it is obtained that 
\begin{align*}
    \biggl|\int_{\R^n} u_i\,  \varphi \,  \d x \bigg |_s \biggr| \leq \| \varphi\|_{\L^{\infty}}\int_{\R^n} u_i \,  \d x \bigg |_s \leq \| \varphi\|_{\L^{\infty}} \biggl(\int_{\R^n} g_i + \int_0^s \int_{\R^n} R_i \, \d x \d t\biggr),
\end{align*}
for every $0\leq s \leq T$. In conclusion, by the Arzel\`{a}–Ascoli theorem, the sequence $\int_{\R^n} u_i \, \varphi \, \d x : [0,T] \to \R$ converges uniformly to a continuous function $f_{\varphi}:[0,T] \to \R$, up to a subsequence. Since $u_i$ converges weakly to $u$ in $\L^1([0,T];\L^1_{\loc}(\R^n))$, up to a subsequence, $f_{\varphi} = \int_{\R^n} u \, \varphi \, \d x$ a.e. in $[0,T].$ By combining the previous results, $\int_{\R^n} u \, \varphi \, \d x$ is continuous in $[0,T]$ outside a set $F_{\varphi}$ of measure zero, which depends on $\varphi$. We define $F := \cup_{\varphi \in D} F_{\varphi} \cup G$, where $[0,T] \setminus G$ is the set of points $s \in [0,T]$ such that $\int_{\R^n} u \, \d x \big |_s$ is finite and $s$ is Lebesgue point for the function $\int_{\R^n} u \, \d x$. By Fubini's theorem and Lebesgue's differentiation theorem, the subset $F$ is of Lebesgue measure zero. Let $s \in (0,T) \setminus F$ and $0<\epsilon < \min (|T-s|,s)$. Then, by \eqref{eq:masconserv} and Fatou's lemma,
\begin{align*}
   \frac{1}{2\epsilon} \int_{s-\epsilon}^{s+\epsilon} \int_{\R^n} u \, \d x \d t
&\leq  \frac{1}{2\epsilon} \liminf_{i \to \infty} \int_{s-\epsilon}^{s+\epsilon} \int_{\R^n} u_i \, \d x \d t
\\ & = \liminf_{i \to \infty} \int_{\R^n} g_i \, \d x + \frac{1}{2\epsilon} \int_{s-\epsilon}^{s+\epsilon} \int_{0}^t \int_{\R^n} R_i \, \d x \d r \d t\\
&=  \int_{\R^n} g \, \d x + \frac{1}{2\epsilon} \int_{s-\epsilon}^{s+\epsilon} \int_{0}^t \int_{\R^n} R \, \d x \d r \d t  .
\end{align*}
Now, by letting $\epsilon \to 0$ in the above inequality and using Lebesgue's differentiation theorem, we arrive at
\begin{align}
\label{eq:fatoutypeineq}
      \int_{\R^n} u \, \d x \bigg |_s \leq \int_{\R^n} g \, \d x + \int_{0}^s \int_{\R^n} R \, \d x \d t.
\end{align}
Let $\psi \in C^{\infty}_0(\R^n)$ and $\varphi_i \in D$ be a sequence of functions converging to $\psi$ in $C_b(\R^n)$. Hence, by \eqref{eq:fatoutypeineq}, for every $s \in [0,T] \setminus F$ we have
\begin{align*}
  \biggl|\int_{\R^n} u  \, \psi \, \d x \, \bigg |_s - \int_{\R^n} g \, \psi \, \d x \biggr| &\leq \biggl|\int_{\R^n} u \, \varphi_i \, \d x \bigg |_s - \int_{\R^n} g \, \varphi_i \, \d x  \biggr| \\&+ \|\psi-\varphi_i \|_{\L ^{\infty}(\R^n)} \biggl(\int_{\R^n} u \, \d x  \biggr |_s + \int_{\R^n} g \, \d x\biggr) \\& \leq \biggl|\int_{\R^n} u \, \varphi_i \, \d x  \bigg |_s - \int_{\R^n} g \, \varphi_i \, \d x \biggr| \\ &+  \|\psi-\varphi_i \|_{\L ^{\infty}(\R^n)} \biggl(2\int_{\R^n} g \, \d x + \int_0^T \int_{\R^n} R \, \d x \d t\biggr).
\end{align*}
By letting $s \to 0$, we derive
\begin{align*}
  \limsup_{s \to 0; s \in [0,T] \setminus F}   \biggl|\int_{\R^n} u \, \psi \, \d x \, \bigg |_s - \int_{\R^n} g \, \psi \, \d x \biggr| \leq  \|\psi-\varphi_i \|_{\L ^{\infty}(\R^n)} \biggl(2\int_{\R^n} g \, \d x + \int_0^T \int_{\R^n} R \, \d x \d t\biggr).
\end{align*}
Since $i$ is arbitrary, it is implied that
\begin{align}
\label{eq:weak*convergence}
    \lim_{s \to 0;[0,T] \setminus F}   \int_{\R^n} u \, \psi \, \d x\,  \bigg |_s = \int_{\R^n} g \, \psi \, \d x.
\end{align}
Now, we prove that in fact the convergence in \eqref{eq:weak*convergence} holds for every $\psi \in C_b(\R^n)$. First, by \eqref{eq:fatoutypeineq}, we obtain
     \begin{align*}
       \limsup_{s \to 0; s \in [0,T] \setminus F}   \int_{\R^n} u \, \d x \bigg |_s \leq \int_{\R^n} g \, \d x.
     \end{align*}
Then, by \eqref{eq:weak*convergence}, we arrive at
\begin{align*}
    \int_{\R^n} g \, \psi_i \, \d x =     \lim_{s \to 0;[0,T] \setminus F}   \int_{\R^n} u \, \psi_i \, \d x\,  \bigg |_s \leq \liminf_{s \to 0;[0,T] \setminus F}\int_{\R^n} u \, \d x\,  \bigg |_s,
\end{align*}
where $\psi_i$ are the increasing sequence of functions in \eqref{eq:spectestfunc}.
Thus, by combining the previous inequalities, we conclude  
\begin{align*}
     \lim_{s \to 0; s \in [0,T] \setminus F}   \int_{\R^n} u \, \d x \bigg |_s = \int_{\R^n} g \, \d x.
\end{align*}
Moreover,
\begin{align*}
    \lim_{s \to 0; s \in [0,T] \setminus F}   \int_{\R^n} u \, (1-\psi_i) \, \d x \bigg |_s &= \lim_{s \to 0; s \in [0,T] \setminus F}   \int_{\R^n} u  \, \d x \bigg |_s -  \int_{\R^n} u  \, \psi_i \,  \d x \bigg |_s \\& = \int_{\R^n} g \, (1-\psi_i) \, \d x.
\end{align*}
Hence,  \begin{align*}
    \lim_{i \to \infty} \limsup_{s \to 0; s \in [0,T] \setminus F} \int_{\R^n \setminus B(0,2i)} u \, \, \d x \bigg |_s  =0.
\end{align*}
In conclusion, by Prokhorov's theorem, we imply \eqref{eq:conticonvinitial} for every $\psi \in C_b(\R^n)$.
\\
\textbf{6}. In this step, we remove the extra assumptions in Step 2 to prove the existence and vague convergence to the initial data. Notice that, by the definition, the convection solutions depend linearly on $R,g.$  Hence, it is enough to find a convection solution for $R^+,g^+.$ By Step 5, for the initial data $R=0, g= e^{-|x|^2}$ there exist a convection solution $ u_e$. Since $\ln (g^+ + e^{-x^2}) \geq - |x|^2$, there exists a convection $u_{eg}$ solution with initial data $R^+,g^+ + e^{-x^2}.$ Taking the difference $u_{eg}- u_e$ gives us the desired result. Also, by Step 5, there exist sets $F_1,F_2 \subset [0,T]$ of Lebesgue measure zero such that 
\begin{align*}
    &\lim_{s \to 0;\, s \in [0,T] \setminus F_1}  \int_{\R^n} u_e \, \psi \, \d x \bigg |_s =\int_{\R^n} e^{-x^2} \, \psi \, \d x  , \\&  \lim_{s \to 0;\, s \in [0,T] \setminus F_2}  \int_{\R^n} u_{eg} \, \psi \, \d x \bigg |_s=\int_{\R^n}(g^+ + e^{-x^2}) \, \psi \, \d x,
\end{align*}
for every $\psi \in C_b(\R^n).$
Hence,
\begin{align*}
      \lim_{s \to 0;\, s \in [0,T] \setminus (F_1 \cup F_2)} \int_{\R^n}(u_{eg}-u_e) \, \psi \, \d x \bigg |_s = \int_{\R^n} g^+ \, \psi \, \d x,
\end{align*}
for every $\psi \in C_b(\R^n)$.\\
\textbf{7}. Finally, we prove that the convection solutions are weak solutions and they satisfy the conservation of mass. Let $\phi \in C^{\infty}_0([0,T] \times \R^n)$ such that $\phi(T,x)=0$ for every $x \in \R^n$. Then, 
\begin{align*}
      \int_0^T \int_{\R^n}- u_i \,\partial_t \phi +  A_i \nabla u_i \cdot \nabla \phi - u_i\, v_i \cdot \nabla \phi \, \d x \d t = \int_{\R^n} g_i \, \phi \, \d x + \int_0^T \int_{\R^n} R_i \, \phi \, \d x \d t.
\end{align*}
Since $u_i$ converges to $u$ pointwise a.e. in $[0,T] \times \R^n$ and weakly in $\L^2([0,T],W^{1,p}_{\loc}(\R^n))$ for $1<p < \frac{n}{n-1},$ up to a subsequence depending on $p$, we derive that
\begin{align*}
    \int_0^T \int_{\R^n} -u \, \partial_t \phi + A \nabla u \cdot \nabla\phi -u\, v \cdot \nabla \phi  \, \d x \d t = \int_{\R^n} g \, \phi \, \d x + \int_0^T \int_{\R^n} R \, \phi \, \d x \d t.
\end{align*}
Thus, $u$ is a weak solution.
 Finally, we prove that $u$ satisfies the conservation of mass if \eqref{eq:derviassum} holds. Let take $F \subset [0,T]$ defined in Step 6 and $s \in [0,T] \setminus F$. Assume that $\psi_i \in C^{\infty}_0(\R^n)$ is the sequence of functions in \eqref{eq:spectestfunc}. 
Define  $$\theta_i(r) := \frac{i}{2} \int_{r}^T 1_{[s-\frac{1}{i},s+\frac{1}{i}]} (t) \, \d t,$$ for every $r \in [0,T].$ Then, by \eqref{eq:derviassum}, we have
\begin{equation*}
    f_i :=  -\partial_t (\theta_i \, \psi_i ) - \div A^* \nabla (\theta_i \, \psi_i) - v \cdot  \nabla (\theta_i \, \psi_i ) \in \L^{\infty}([0,T] \times \R^n) \cap \L^2([0,T] \times \R^n),
\end{equation*}
weakly in $[0,T] \times \R^n$ and $\theta_i \,\psi_i (T,x)=0$ for every $x \in \R^n.$ Hence,
\begin{align*}
   \int_{0}^T \int_{\R^n} f_i\,  u \, \d x \d t = \int_{\R^n} \theta_i \, \psi_i \,  g \, \d x \bigg |_0 + \int_{0}^T \int_{\R^n} \theta_i \, \psi_i \, R \, \d x \d t.
 \end{align*}
Taking $i \to \infty$ above and using Lebesgue's dominated convergence and Lebesgue's density theorems, we derive that 
\begin{align*}
    \int_{\R^n} u \, \d x \, \bigg |_s = \int_{\R^n} g  \, \d x + \int_0^s \int_{\R^n} R \, \d t  \d x, 
\end{align*}
which completes the proof.

\end{proof}

    

\begin{exm}
\label{exm:example}
   We use the idea in \cite{Prignet,Serrin} to demonstrate the non-uniqueness of weak solutions for coefficients that do not satisfy \eqref{eq:derviassum}. Assume that $a_{ij}(t,x)= \delta_{ij} + (a-1) \frac{x_i x_j}{|x|^2}$ for $(t,x) \in [0,T] \times \R^n$ and $n=2,3$, where the positive constant $a$ is to be determined later. Then, by a classical result of Serrin, see \cite{Serrin}, the function $u(x) = x_1 |x|^{1-n-\epsilon}$ is a weak solution for $-\div A \nabla u=0$ in $\R^n$, where $0<\epsilon<\frac{1}{2}$ and $a = \frac{n-1}{\epsilon (\epsilon + n -2)}.$ We note that $\nabla u  \notin W^{1,2}_{\loc}(\R^n)$ and $\nabla u  \in W^{1,p}_{\loc}(\R^n)$ for every $ p< \frac{n}{n+\epsilon-1}$. Now, let $\varphi \in C^{\infty}_0(\R^n)$ be a cut-off function such that $\varphi = 1$ in the unit ball $B(0,1)$ and define $R := -\div A \nabla (u\, \varphi)$. Since $-\div A \nabla u=0$ in $\R^n$ weakly and $A,u$ are smooth in $\R^n \setminus 0$, we obtain $R \in C^{\infty}_0(\R^n)$. If we define $g := u\, \varphi$, then $u\, \varphi$ is the weak solution of $\partial_t (u\, \varphi) -\div A \nabla (u\, \varphi) = R$ in $[0,T] \times \R^n$ with initial data $g$, and $u\, \varphi$ does not belong to $\L^2([0,T]; W^{1,2}_{\loc}(\R^n))$. Since $g \in \L^2(\R^n), R \in \L^2([0,T] \times \R^n)$, by variational method in \cite{LM} or Galerkin method in \cite{E}, there exists a solution $\Tilde{u}_i \in \L^2([0,T]; W^{1,2}_0(B(0,i)))$ which satisfies $\partial_t \Tilde{u}_i -\div A \nabla \Tilde{u}_i = R$ weakly in $[0,T] \times B(0,i)$ with initial data $g$. By standard regularity estimates, we derive 
$$a \|\Tilde{u}_i\|^2_{\L^2([0,T]; W^{1,2}(B(0,i)))} \leq \|R\|^2_{\L^2([0,T] \times B(0,i))}+ \|g\|^2_{\L^2(B(0,i))}.$$ Hence, up to a subsequence, $\Tilde{u}_i$ converges weakly to $\Tilde{u}$ in $\L^2([0,T]; W^{1,2}(\R^n))$ which satisfies $\partial_t \Tilde{u} -\div A \nabla \Tilde{u} = R$ weakly in $[0,T] \times \R^n$ with initial data $g$. Moreover, $w := u-\tilde{u}$ belongs to $\L^1([0,T];\W^{1,1}_{\loc}(\R^n))$ and satisfies $\partial_t w- \div A \nabla w=0$ weakly in $[0,T] \times \R^n$ with the initial data $0$ but $w \neq 0.$

\end{exm}

\section{Existence for drifts induced by potentials}
We assume that $V: \R^n \to \R$ is a measurable function, where $e^{V} + e^{-V} \in \L^1_{\loc}(\R^n)$ and 
$e^{-V}$ is a regular weight. The following examples are the main motivations:
\begin{exm}
    If $V \in \L^{\infty}_{\loc}(\R^n)$, then $e^{V} + e^{-V} \in \L^1_{\loc}(\R^n)$. To prove that $e^{-V}$ is a regular weight consider $f \in \H^1_{e^{-V}}(\R^n)$ and $\psi_i \in C^{\infty}_0(\R^n)$ constructed in \eqref{eq:spectestfunc}. Then,
\begin{align*}
    \|(1-\psi_i) f\|^2_{\H^1_{e^{-V}}(\R^n)}
    \leq \int_{\R^n} \biggl((1-\psi_i)^2 + \frac{C_1^2}{j^2} \biggr) f^2 + (1-\psi_i)^2 |\nabla f|^2 \, \d e^{-V}.
\end{align*}
Hence, by Lebesgue's dominated convergence $\lim_{i \to \infty} \|(1-\psi_i) f\|_{\H^1_{e^{-V}}(\R^n)}=0.$ Moreover, since $\psi_i f$ is compactly supported and $e^{-V}$ is comparable with Lebesgue measure on compact subsets of $\R^n$, there exists $f_{i} \in C^{\infty}_0(\R^n)$ such that $\|f_{i} - \psi_i f\|_{\H^1_{e^{-V}}(\R^n)} \leq \frac{1}{i}.$ In conclusion, by triangle inequality, we derive that $\lim_{i \to \infty} \|f_{i} - f\|_{\H^1_{e^{-V}}(\R^n)}=0$.

\end{exm}
\begin{exm}
   Let $V(x)= C \ln |x|$ for $x \in \R^n$ and a constant $-n<C<n$. Then, $e^{-V}$ belongs to Muckenhoupt weight $A_2(\R^n)$, see \cite[Ch. 15]{HKM}. Hence, $e^{V} + e^{-V} \in \L^1_{\loc}(\R^n)$ and, by \cite[Thm. 15.21]{HKM}, $e^{-V}$ is a regular weight. 
\end{exm}

We need the following concept of an adjoint convection solution.
\begin{defn}\label{def:dualsolpoten}
Let $f \in  \L^{\infty}([0,T] \times \R^n) \cap \L^2_{e^{-V}}([0,T] \times \R^n).$ We define the adjoint convection solution $\phi \in \L^2([0,T];\H^1_{e^{-V}}(\R^n)) \cap C([0,T];\L^2_{e^{-V}}(\R^n))$ of
\begin{align}
    \label{eq:dualsol2}
    -\partial_t \phi -e^V \div (e^{-V} \nabla \phi) =f,
\end{align}
in $[0,T] \times \R^n$ if  $\phi(T,x)=0$ for a.e. $x \in \R^n$ and
\begin{align*}
&\int_{\R^n} \phi \, \psi\, \d e^{-V} \bigg |_0  + \int_{0}^T \int_{\R^n} \phi \, \partial_t \psi \, \d e^{-V} \d t \\& + \int_{0}^T \int_{\R^n} \nabla \phi \cdot \nabla \psi \, \d e^{-V} \d t = \int_0^T \int_{\R^n} f\, \psi\, \d e^{-V} \d t, 
\end{align*}
for every $\psi \in C^{\infty}_0([0,T] \times \R^n ).$
\end{defn}

We need the following auxiliary lemma.
\begin{lem}
\label{lem:approxpoten}
    There exists a sequence $V_i \in C^{0,1}(\R^n) \cap C^{\infty}(\R^n)$ such that $\inf_{\R^n} V_i > -\infty$, $V_i$ converges pointwise to $V$, and $e^{V_i},e^{-V_i}$ converge to $e^V,e^{-V}$ in $\L^1(B)$, respectively, for every ball $B \subset \R^n.$ Moreover, if $e^{-V} \in \L^1(\R^n)$, then $e^{-V_i}$ converges to $e^{-V}$ in $\L^1(\R^n).$
\end{lem}
\begin{proof}
    Let $\varphi_{\epsilon}$ be a positive approximation of identity, where $\epsilon$ is a positive constant.  Define $V_{\epsilon,\theta,\delta}(x) :=  \varphi_{\epsilon} \ast (V 1_{B(0,\frac{1}{\theta})})(x) + \delta \sqrt{|x|^2+1} \in C^{0,1}(\R^n) \cap C^{\infty}(\R^n)$ for every $x \in \R^n$, where $\theta,\delta$ are positive constants. Then, $\inf_{\R^n} V_{\epsilon,\theta,\delta} > -\infty.$ Since $e^V+ e^{-V} \in \L^1_{\loc}(\R^n)$, we obtain that $e^{|V|} \in \L^1_{\loc}(\R^n).$ Hence, by $ \frac{|V|^n}{n!} \leq e^{|V|}$ for every positive integer $n$, we derive that $V \in \L^{p}_{\loc}(\R^n)$ for every $p>0.$ Now, we prove that \begin{align}
    \label{eq:locconverg}
    \lim_{(\epsilon,\theta,\delta) \to 0}  \|e^{V_{\epsilon,\theta,\delta}} -e^{V}\|_{\L^1(B)} +  \|e^{-V_{\epsilon,k,\delta}} -e^{-V}\|_{\L^1(B)} =0,
        \end{align} for every ball $B \subset \R^n.$ Let us fix a ball $B \subset \R^n.$ Since $B$ is bounded, 
        $$e^{V_{\epsilon,\theta,\delta}} + e^{-V_{\epsilon,k,\delta}} \bigg |_B \leq C_B e^{\varphi_{\epsilon} \ast (|V| 1_{B(0,\frac{1}{\theta})})} \bigg |_B,$$ for constant $C_B$ depending on $B.$
    By a variant of Lebesgue's dominated convergence, see \cite[Thm. 1.20]{EG}, it is enough to prove that
\begin{align}
\label{eq:upperlimconv}
    \lim_{(\epsilon,\theta) \to 0}\|e^{\varphi_{\epsilon} \ast (|V| 1_{B(0,\frac{1}{\theta})})} - e^{|V|}\|_{\L^1(B)} =0.
\end{align}
By Young's convolution inequality, we obtain 
    \begin{align*}
        \int_{B} |\varphi_{\epsilon} \ast (|V| 1_{B(0,\frac{1}{\theta})})|^p  \, \d x \leq \int_{\R^n} |\varphi_{\epsilon} \ast (|V| 1_{(1+\epsilon)B})|^p  \, \d x  \leq  \int_{(1+\epsilon )B} |V|^p \, \d x.
    \end{align*}
    for every $p \geq 1.$  Hence, by the monotone convergence theorem, it is obtained that
\begin{align*}
    \int_{B} e^{\varphi_{\epsilon} \ast (|V| 1_{B(0,\frac{1}{\theta})})} \, \d x &= \sum_{j=1}^{\infty} \int_{B} \frac{|\varphi_{\epsilon} \ast (|V| 1_{B(0,\frac{1}{\theta})})|^j}{j!} \, \d x  \\&  \leq \sum_{j=1}^{\infty} \int_{(1+\epsilon)B} \frac{|V |^j}{j!} \, \d x = \int_{(1+\epsilon)B} e^{|V|} < \infty.
\end{align*}
Also, by Fatou's lemma, we have
\begin{align*}
   \int_{B} e^{|V|} \, \d x \leq \liminf_{(\epsilon,\theta) \to 0} \int_{B} e^{\varphi_{\epsilon} \ast (|V| 1_{B(0,\frac{1}{\theta})})} \, \d x.
\end{align*}
Hence, $\lim_{(\epsilon,\theta) \to 0} \int_{B} e^{\varphi_{\epsilon} \ast (|V| 1_{B(0,\frac{1}{\theta})})} \, \d x =  \int_{B} e^{|V|} \, \d x$. Then, since $e^{\varphi_{\epsilon} \ast (|V| 1_{B(0,\frac{1}{\theta})})} $ converges ponitwise to $e^{|V|}$, by Riesz-Scheff\'e-Brezis–Lieb lemma we obtain $e^{\varphi_{\epsilon} \ast (|V| 1_{B(0,\frac{1}{\theta})})}$ converges to $e^{|V|}$ in $\L^1(B)$ as $(\epsilon,\theta)$ goes to $0$. This completes the proof of \eqref{eq:locconverg}. Now, if $e^{-V} \in \L^1(\R^n)$, we construct a sequence of $\epsilon_i,\theta_i,\delta_i$ converging to zero such that
\begin{align}
\label{eq:potentlimit}
     \lim_{i \to \infty} \|e^{-V_{\epsilon_i,\theta_i,\delta_i}} -e^{-V}\|_{\L^1(\R^n)} =0
\end{align}
Since $e^{-V_{\epsilon,\theta,\delta}} \leq 1_{B(0,\frac{1+\epsilon}{\theta})} e^{\varphi_{\epsilon} \ast (|V| 1_{B(0,\frac{1}{\theta})})}  + e^{-\delta \sqrt{|x|^2+1}}$ and $\int_{\R^n} e^{-\delta \sqrt{|x|^2+1}} < \infty$, by \eqref{eq:upperlimconv} and a variant of Lebesgue's dominated convergence, see \cite[Thm. 1.20]{EG}, we have
\begin{align}
\label{eq:firstlimit}
    \lim_{\epsilon \to 0} \|e^{-V_{\epsilon,\theta,\delta}}-e^{-V 1_{B(0,\frac{1}{\theta})} - \delta \sqrt{{|x|^2}+1}}\|_{\L^1(\R^n)} =0.
\end{align}
    Also, by Lebesgue's dominated convergence and $e^{-V 1_{B(0,\frac{1}{\theta})} - \delta \sqrt{{|x|^2}+1}} \leq e^{-V}+ e^{- \delta \sqrt{{|x|^2}+1}}$, we deduce
\begin{equation}
\begin{aligned}
\label{eq:secondlimit}
  &  \lim_{\theta \to 0} \|e^{-V 1_{B(0,\frac{1}{\theta})} - \delta \sqrt{{|x|^2}+1}}-e^{-V- \delta \sqrt{{|x|^2}+1}}\|_{\L^1(\R^n)} =0,\\
    &\lim_{\delta \to 0} \|e^{-V - \delta \sqrt{{|x|^2}+1}}-e^{-V}\|_{\L^1(\R^n)} =0.
\end{aligned}
  \end{equation}
In conclusion, by triangle inequality, \eqref{eq:firstlimit}, and \eqref{eq:secondlimit}, there exists a sequence $(\epsilon_i,\theta_i,\delta_i)$ converging to $0$ such that \eqref{eq:potentlimit} holds.
  
\end{proof} For simplicity, define the weights $$w_i := e^{-V_i},\quad w :=e^{-V},$$ where $V_i$ is the sequence in Lemma \ref{lem:approxpoten}. We now construct an approximating sequence of weak solutions for \eqref{eq:dualsol2}.
\begin{prop}
\label{prop:advecsol}
For every $f \in \L^{\infty}([0,T] \times \R^n)\cap \L^2([0,T];\L^2_w( \R^n))$, there exists a unique adjoint convection solution $\phi \in \L^2([0,T];\H^1_w(\R^n)) \cap C([0,T];\L^2_{w}(\R^n))$ of \eqref{eq:dualsol2} in $[0,T] \times \R^n$.  Also, $\phi$ satisfies the following:\\
(i). For every uniformly bounded sequence $f_i \in \L^{\infty}([0,T] \times \R^n)\cap \L^2([0,T];\L^2_w( \R^n))$ which converges to $f$ in $\L^2_w([0,T] \times \R^n)$, there exists a sequence $\phi_i$ belonging to $\L^2([0,T];\H^1_w(\R^n)) \cap C([0,T];\L^2_{w}(\R^n))$ which are adjoint convection solutions of
    \begin{align}
    \label{eq:approxdualpotential}
           -\partial_t \phi_i - e^{V_i} \div( e^{-V_i} \nabla \phi_i) =f_i,
    \end{align} in $[0,T] \times \R^n$, and
    $\phi_i$ converges to $\phi$ in $C([0,T];\L^1_{\loc}(\R^n))$.
    \\(ii).
    \begin{align*}
 &   (T-t) \inf f  \leq \phi(t,\cdot) \leq  (T-t) \sup f,\\
& \sup_{0 \leq t \leq T} \frac{1}{p} \|\phi(t,\cdot)\|^p_{\L^{p}_{w}(\R^n)} + \|
 \phi^{\frac{p-2}{2}}\nabla \phi \|^2_{\L^2_{w}( [0,T] \times \R^n)} \leq  
 \frac{e^{(p-1)T}}{p}\,  \|f\|^p_{\L^{p}_{w}([0,T] \times \R^n)},\\ 
 &\sup_{0\leq t \leq T-h}\|\tau_h \phi(t,\cdot) - \phi(t,\cdot)\|^2_{\L^2_w(\R^n)} \leq e^T\biggl(
\|\tau_h f-f \|^2_{\L^2_{w}([0,T-h] \times \R^n)}+ \int_{T-h}^T \int_{\R^n} |f|^2 \, \d w \biggr),
    \end{align*}
for every $p>1$, where $f \in \L^{p}_{w}(\R^n)$, and $0<h<T$. 
\end{prop}
\begin{proof}
For the uniqueness, it is enough to prove that $0$ is the only adjoint convection solution of 
\eqref{eq:dualsol2} if $f=0.$ Let $\phi \in \L^2([0,T];\H^1_{w}(\R^n)) \cap C([0,T];\L^2_{w}(\R^n))$ be an adjoint convection solution of \begin{align}
\label{eq:zerodualsol}
    -\partial_t \phi -e^V \div (e^{-V} \nabla \phi) =0,
\end{align} in $[0,T] \times \R^n$. We assume that $0<\epsilon<T$, $\eta_{\epsilon} \in C^{\infty}_0(\R)$ is an approximation of identity, and $\beta \in C^{\infty}_0((\epsilon,T-\epsilon))$. Then, for every $\psi \in C^{\infty}_0([0,T] \times \R^n)$ we have 
\begin{align*}
   \int_{0}^T \int_{\R^n} \phi \, \partial_t (\eta_{\epsilon} \ast  (\beta \psi))\, \d w \d t  + \int_{0}^T \int_{\R^n} \nabla \phi \cdot \nabla (\eta_{\epsilon} \ast  (\beta \psi)) \, \d w \d t = 0,
\end{align*}
where we used $\eta_{\epsilon}\ast (\beta \psi)(t,x)=0$ if $t=0,T$ and $x \in \R^n.$ Hence, 
\begin{align*}
 &\int_{0}^T \int_{\R^n} - \beta \, \psi \, \partial_t (\eta_{\epsilon} \ast \phi)  \beta \, \psi \, \d w \d t \\& + \int_{0}^T \int_{\R^n} \beta\, \nabla (\eta_{\epsilon} \ast \phi) \cdot \nabla  \psi \, \d w \d t =0
\end{align*}
for every $\psi \in C^{\infty}_0([0,T] \times \R^n)$. Since $w$ is a regular weight and $\phi \in \L^2([0,T];\H^1_w(\R^n))$, we can take an approximation of $\eta_{\epsilon} \ast \phi$ by elements in $C^{\infty}_0([\epsilon,T-\epsilon] \times \R^n)$ in $\L^2([\epsilon,T-\epsilon]; \H^1_w(\R^n))$ to obtain 
\begin{align*}
  &  \int_{0}^T \int_{\R^n} - \beta \,  \partial_t (\eta_{\epsilon} \ast \phi) \,  \eta_{\epsilon} \ast \phi \, \d w \d t \\& + \int_{0}^T \int_{\R^n} \beta \, \nabla \eta_{\epsilon} \ast \phi \cdot \nabla  \eta_{\epsilon} \ast \phi \, \d w \d t =0.
\end{align*}
Thus,
\begin{align*}
   \frac{1}{2}  \int_{0}^T \int_{\R^n}  (\eta_{\epsilon} \ast \phi)^2  \partial_t \beta \, \d w \d t  + \int_{0}^T \int_{\R^n} \beta\, \nabla \eta_{\epsilon} \ast \phi \cdot \nabla  \eta_{\epsilon} \ast \phi \, \d w \d t =0.
\end{align*}
Letting $\epsilon \to 0$ in the above derives that
\begin{align}
\label{eq:zeroequation}
    \frac{1}{2} \int_{0}^T \int_{\R^n}   \phi^2  \, \partial_t \beta \, \d w \d t + \int_{0}^T \int_{\R^n} \beta \, |\nabla \phi|^2\, \d w \d t =0
\end{align}
for every $\beta \in C_0^{\infty}((0,T))$. By approximation and the assumption $\phi \in C([0,T]; \L^2_w(\R^n)) \cap \L^2([0,T]; \H^1_w(\R^n))$, we can assume that $\beta \in C^{0,1}([0,T])$ and $\beta(0)=\beta(T)=0$ in \eqref{eq:zeroequation}. Define 
\begin{align*}
    \beta_i(t) := \begin{cases}
        0 \quad  &0 \leq t \leq \frac{1}{i},\\
        i t-1 \quad &\frac{1}{i} \leq t \leq \frac{2}{i},\\
        1 \quad &\frac{2}{i} \leq t \leq 1-\frac{2}{i},\\
        i (T-t) -1 \quad & T-\frac{2}{i} \leq t \leq T-\frac{1}{i},\\
        0 \quad & T-\frac{1}{i} \leq t \leq T,
    \end{cases}
\end{align*}
for $i > \frac{1}{T} +1.$ Then, by definition $\beta_i \in C^{0,1}([0,T])$ and $\beta_i(0)=\beta_i(T)=0$. Hence, by using $\phi \in C([0,T]; \L^2_w(\R^n)) \cap \L^2([0,T]; \H^1_w(\R^n))$ and replacing $\beta $ by $\beta_i$ in \eqref{eq:zeroequation}, we deduce
\begin{align*}
   &\int_{\R^n}   \phi^2 \d w \bigg |_0 - \int_{\R^n} \phi^2 \, \d w \bigg |_T + \int_{0}^T \int_{\R^n}  |\nabla \phi|^2 \, \d w \d t \\&= \lim_{i \to \infty}   \int_{0}^T \int_{\R^n}   \phi^2 \, \partial_t \beta_i \, \d w \d t + \int_{0}^T \int_{\R^n} \beta_i \, |\nabla \phi|^2\, \d w \d t =0.
\end{align*}
Hence, $\nabla \phi =0$ and $\phi$ is a function of time. By \eqref{eq:zerodualsol}, we imply that $-\partial_t \phi =0$ in $[0,T] \times \R^n$. Since $\phi(T,x)=0$ for a.e. $x \in \R^n,$ we conclude that $\phi =0.$
Now, to prove the existence part, let $f_i \in C^{\infty}_0([0,T] \times \R^n)$ where $f_i$ is a uniformly bounded sequence which converges to $f$ in $\L^2([0,T];\L^2_w( \R^n))$. By the standard theory of parabolic equations, see \cite{A}, there exists a smooth weak solution $\phi_i \in \L^2([0,T]; \H^1(\R^n))$ satisfying \eqref{eq:approxdualpotential} and $\phi_i(T,x)=0$ for every $x \in \R^n.$ Then, by Proposition \ref{prop:dualsol}, we have
\begin{align}
\label{eq:maximumprin}
(T-t) \inf f_i \leq \phi_i(t,\cdot) \leq (T-t) \sup f_i,    
\end{align}
for every $0 \leq t \leq T$. Let $0 \leq T' \leq T$. Using $\phi_i\, |\phi_i|^{p-2}\, w_i$, for $p>1$, as a test function gives
\begin{align*}
 \partial_t \frac{1}{p} \int_{\R^n} |\phi_i|^p \,  \d w_i \bigg |_t +  p
 \int_{\R^n} |\phi_i|^{p-2} |\nabla \phi_i|^2 \, \d w_i \d t \bigg |_t \leq  \int_{\R^n} |f_i| \,|\phi_i|^{p-1} \, \d w_i \bigg |_t.
 \end{align*}
Hence, by Young and Gr\"{o}nwall's inequality, we obtain
\begin{equation}
\label{eq:energyestim}
\begin{aligned} 
 &\sup_{T' \leq t \leq T} \frac{1}{p} \int_{\R^n} |\phi_i|^p \,\d w_i \bigg |_t +  p \,e^{-(p-1)T}
\int_{T'}^T \int_{\R^n} |\phi_i|^{p-2} |\nabla \phi_i|^2 \, \d w_i \d t \\& \leq \frac{e^{(p-1)T} }{p} \int_{T'}^T \int_{\R^n} |f_i|^{p} \, \d w_i \d t,
\end{aligned}
 \end{equation} for every $p>1.$ Now, by definition $\Phi_i := \tau_h \phi_i - \phi_i$ is the weak solution of 
\begin{align*}
     -\partial_t \Phi_i - \Delta \Phi_i + \nabla V_i \cdot \nabla \Phi_i =\tau_h f_i - f_i,
\end{align*}
in $[0,T-h] \times \R^n$ with $\Phi_i(T-h,x) = -\phi_i(T-h,x)$ for $x \in \R^n,$ where the shift operator $\tau_h$ is defined in Section \ref{sec:spaceandmeas}. Taking $\psi_i\, w_i$ as a test function and using Young's inequality as above, we obtain
\begin{align}
\label{eq:differeequ}
  \sup_{0 \leq t \leq T-h}  \int_{\R^n} \Phi_i^2 \, \d w_i \bigg |_t \leq e^T \int_0^{T-h} \int_{\R^n} |\tau_h f_i - f_i|^2\, \d w_i \d t + \int_{\R^n} \Phi_i^2(T-h,x)   \, \d w_i. 
\end{align}
Let $B \subset \R^n$ be a fixed ball and $0 \leq t_1 < t_2 \leq T$. Then, by Cauchy-Schwarz inequality and \eqref{eq:energyestim}, it is implied that
 \begin{align*}
    & \biggl ( \sup_{0 \leq t \leq T-h} \int_{B} |\tau_h \phi_i - \phi_i| \, \d x \bigg |_t  \biggr )^2 \leq  w_i^{-1}(B)  \sup_{0 \leq t \leq T-h} \int_{B} |\tau_h \phi_i - \phi_i| ^2 \, \d w_i \bigg |_t  \\& \leq  e^T w_i^{-1}(B) \biggl( \int_0^{T-h} \int_{\R^n} |\tau_h f_i - f_i|^2 \, \d w_i \d t +\int_{T-h}^T \int_{\R^n} |f_i|^2 \, \d w_i \d t  \biggr).
 \end{align*}
 In conclusion, $\lim_{h \to 0}\sup_{i} \|\tau_h \phi_i - \phi_i\|_{\L^1([0,T-h] \times B)} =0$. Moreover, by \eqref{eq:energyestim} and Cauchy-Schwarz inequality, we have
 \begin{align*}
  \biggl ( \int_{B}  \biggl|\nabla \int_{t_1}^{t_2} \phi_i \, \d t \biggr| \, \d x \biggr)^2 & \leq  (t_2-t_1) \, w_i^{-1}(B)\int_{B}  \int_{t_1}^{t_2} |\nabla \phi_i|^2 \, \d t  \d w_i \\& \leq T e^T  \, w_i^{-1}(B) \int_{B}  \int_{t_1}^{t_2} f_i^2 \, \d t  \d w_i
 \end{align*}
Thus, by Sobolev embedding, $\int_{t_1}^{t_2} \phi_i \, \d t$ is relatively compact in $\L^1( B)$. Now, by the classical result in \cite[Thm. 1]{Simon}, $\phi_i$ is relatively compact in $C([0,T]; \L^1(B)).$ Hence, up to a subsequence, $\phi_i$ converges to $\phi \in \L^{\infty}([0,T]\times \R^n) \cap C([0,T]; \L^1_{\loc}(\R^n))$ in $C([0,T]; \L^1_{\loc}(\R^n))$ and $\L^{\infty}([0,T] \times \R^n)$-weak*, where $\phi(T,x)=0$ for a.e. $x \in \R^n$. 
 By Lemma \ref{lem:approxpoten} and \eqref{eq:energyestim}, for every ball $B \subset \R^n$, we have $w_i^{\frac{1}{2}} \nabla \phi_i$ is uniformly bounded in $\L^2([0,T] \times \R^n,\R^n)$ and $w_i^{\frac{1}{2}}, w_i^{-\frac{1}{2}}$ converges to $w^{\frac{1}{2}}, w^{-\frac{1}{2}}$, respectively, in $\L^2_{\loc}(\R^n)$. Since, $\phi_i$ converges to $\phi$ in $\L^{\infty}([0,T] \times \R^n)$-weak*, we conclude that  
 $w_i^{\frac{1}{2}} \nabla \phi_i$ converges to $w^{\frac{1}{2}} \nabla \phi$ weakly in $\L^2([0,T] \times \R^n,\R^n)$. In fact,  if  $w_i^{\frac{1}{2}} \nabla \phi_i$ converges to $F=(F_1,\cdot \cdot \cdot, F_n)$ weakly in $\L^2([0,T] \times \R^n,\R^n)$, then
\begin{align*}
  \lim_{i \to \infty} \int_{0}^T \int_{\R^n} \Phi\,  w_i^{\frac{1}{2}} \partial_{x_j} \phi_i \, \d x \d t= \int_{0}^T \int_{\R^n} \Phi\, F_j \, \d x \d t,
\end{align*}
for every $1 \leq j \leq n$ and $\Phi \in \L^2([0,T] \times \R^n)$. Now, by Hölder's inequality and the convergence of $w_i^{-\frac{1}{2}}$ to $w^{-\frac{1}{2}}$ in $\L^2(\R^n)$, we obtain
\begin{align*}
  \lim_{i \to \infty} \int_{0}^T \int_{\R^n} w_i^{^{-\frac{1}{2}}} \varphi \cdot w_i^{\frac{1}{2}} \partial_{x_j} \phi_i \, \d x \d t= \int_{0}^T \int_{\R^n} \varphi \, F_j \, \d x \d t,
\end{align*}
for every $1 \leq j \leq n$ and $\varphi \in C^{\infty}_0([0,T] \times \R^n)$. Then, by the convergence of $\phi_i$ to $\phi$ in $\L^{\infty}([0,T] \times \R^n)$-weak*, we arrive at $F = w^{\frac{1}{2}} \nabla \phi.$ Moreover, using again \eqref{eq:energyestim} and Lemma \ref{lem:approxpoten}, we derive that $w_i^{\frac{1}{2}}\, \phi_i $ converges to $w^{\frac{1}{2}} \, \phi$ weakly in $\L^2([0,T] \times \R^n)$. Now, we prove that $\phi \in C([0,T]; \L^2_w(\R^n)).$
 By Fatou's lemma, \eqref{eq:energyestim}, and \eqref{eq:differeequ}, we obtain
 \begin{align*}
     \sup_{0 \leq t \leq T-h} \int_{\R^n} |\tau_h \phi - \phi|^2 \, \d w \bigg |_t
     \leq e^T \biggl(\int_0^{T-h} \int_{\R^n} |\tau_h f - f|^2 \, \d w \d t + \int_{T-h}^T \int_{\R^n} f^2 \, \d w \d t\biggr),
 \end{align*}
which proves that $\phi \in C([0,T];\L^2_w(\R^n)).$ Then, using \eqref{eq:approxdualpotential}, \eqref{eq:maximumprin}, and the weak convergence of $w_i^{\frac{1}{2}} \nabla \phi_i,w_i^{\frac{1}{2}} \phi_i $ to $w^{\frac{1}{2}} \nabla \phi,w^{\frac{1}{2}}\phi$, respectively, we imply that $\phi$ is a weak solution of \eqref{eq:dualsol2} with $\phi(T,x)=0$ a.e. in $\R^n$. Since $\phi$ is unique, we derive that $\phi_i$ converges to $\phi$ in $C([0,T];\L^1_{\loc}(\R^n))$ without any need to pass to a subsequence.  To complete the proof of part (i), we note that the proof above applies to any 
 $f_i \in \L^{\infty}([0,T] \times \R^n)\cap \L^2_w([0,T] \times \R^n)$ converging to $f$ in $\L^{\infty}([0,T] \times \R^n)$-weak* and $\L^2_w([0,T] \times \R^n)$. Finally, recall that, $\phi_i$ converges to $\phi$ in $C([0,T];\L^1_{\loc}(\R^n))$. Hence, part (ii) is obtained by applying Fatou's lemma in \eqref{eq:maximumprin},   \eqref{eq:energyestim}, and in \eqref{eq:differeequ}.
\end{proof}
The following lemma derives some properties of the adjoint convection solution if the right-hand side is time-independent. 
\begin{lem} \label{lem:timederiv}
Let $f \in C^{\infty}_0(\R^n)$ and $\phi$ be the adjoint convection solution of \eqref{eq:dualsol2} in the set $[0,T] \times \R^n$. Then, for every $0 \leq s \leq T$, $\phi^s(t,x) := \phi(t+ T-s,x)$, for $0 \leq t \leq s, x \in \R^n$, is the adjoint convection solution for
\begin{align*}
    -\partial_t \phi^s -e^V \div (e^{-V} \nabla \phi^s)= f,
\end{align*}
in $[0,s]  \times \R^n.$ Moreover, for every  weak solution $\phi_i \in C^{\infty}([0,T] \times \R^n) \cap \L^2([0,T];\H^1(\R^n))$ of $-\partial_t \phi_i - \Delta \phi_i + \nabla V_i \cdot \nabla \phi_i=0$ with $\phi_i(T,\cdot)=0$, we have $\partial_t \phi_i$ converges to $\partial_t \phi$ in $C([0,T];\L^1_{\loc}(\R^n))$, $\partial_t \phi(T,x) = -f(x)$ for a.e. $x \in \R^n,$ and 
\begin{align*}
    \|\partial_t \phi\|_{\L^{\infty}} \leq \|f\|_{\L^{\infty}}.
\end{align*}
\end{lem}
\begin{proof}
    The proof uses the argument in Proposition \ref{prop:advecsol}. Let $V_i$ be as in Lemma \ref{lem:approxpoten} and $\phi_i \in \L^2([0,T];\H^1(\R^n))$ be the smooth weak solution for 
\begin{align*}
    -\partial_t \phi_i -\Delta \phi_i + \nabla V_i \cdot \nabla \phi_i = f,
\end{align*}
in $[0,T] \times \R^n$ with $\phi_i(T,x)=0$ for every $x \in \R^n$. Then, $\phi^s_i(t,x) := \phi_i(t+ T-s,x)$ for $0 \leq t \leq s,x \in \R^n$ is the adjoint convection solution of \begin{align*}
    -\partial_t \phi^s_i - e^{V_i} \div(e^{-V_i} \nabla \phi^s_i) = f,
\end{align*}
in $[0,s] \times \R^n.$ Since $\phi^s_i$ converges pointwise to $\phi^s$, by Proposition \ref{prop:advecsol}, $\phi^s$ is the adjoint convection solution of \begin{align*}
    -\partial_t \phi^s  -e^V \div(e^{-V} \nabla \phi^s) = f,
\end{align*}  
    in $[0,s] \times \R^n.$ 
For the second part, we have
\begin{align}
\label{eq:timederiequa}
     -\partial_t (\partial_t \phi_i) -\Delta (\partial_t \phi_i) + \nabla V_i \cdot \nabla (\partial_t \phi_i) = 0,
\end{align}
weakly in $[0,T] \times \R^n$ and 
\begin{align}
\label{eq:timederivendpoin}
    \partial_t \phi_i(T,x) = -\Delta \phi_i(T,x) + \nabla V_i(x) \cdot \nabla \phi_i(T,x) -f(x) = -f(x)
\end{align}
for every $x \in \R^n.$ Hence, by the maximum principle, see \cite{A}, we arrive at
\begin{align}
    \label{eq:Linfboutimederiv}
    \|\partial_t \phi_i\|_{\L^{\infty}} \leq \|f\|_{\L^{\infty}}.
\end{align}
Define $\Phi_i:= \partial_t \phi_i -f$. Then, $\Phi_i$ is the weak solution of 
\begin{align}
\label{eq:mainderieq}
    -\partial_t \Phi_i -\Delta  \Phi_i + \nabla V_i \cdot \nabla \Phi_i = \Delta f-\nabla V_i \cdot \nabla f,
\end{align}
in $[0,T] \times \R^n$ with $\Phi_i(T,x)=0$ for every $x \in \R^n$. Hence, by taking $\Phi_i \, w_i $ as a test function for \eqref{eq:mainderieq}, we obtain that
\begin{align*}
    - \frac{1}{2}\partial_t \int_{\R^n} \Phi_i^2 \, \d w_i \bigg |_t + \int_{\R^n} |\nabla \Phi_i|^2 \, \d w_i = -\int_{\R^n}\nabla f \cdot \nabla \Phi_i \, \d w_i. 
\end{align*}
In conclusion, by Young's inequality, we derive that
\begin{align}
\label{eq:L2estim}
  \sup_{T' \leq t \leq T}  \int_{\R^n} |\Phi_i|^2 \, \d w_i \bigg |_t + \int_{T'}^T \int_{\R^n} |\nabla \Phi_i|^2 \, \d w_i \d t \leq \int_{T'}^T \int_{\R^n} |\nabla f|^2 \, \d w_i \d t,
\end{align}
for every $0 \leq T' \leq T$. Let $0<h < T$ and $\Psi^h_i := \tau_h \Phi_i - \Phi_i$. Then, $\Psi^h_i(T-h,x) = -\Phi_i(T-h,x)$ for every $x \in \R^n$ and 
\begin{align}
\label{eq:diffeqtimederi}
    -\partial_t \Psi^h_i - \Delta \Psi^h_i + \nabla V_i \cdot \nabla \Psi^h_i =0,
\end{align}
in $[0,T-h] \times \R^n.$ By taking the test function $\Psi_i^h \, w_i$ for \eqref{eq:diffeqtimederi} and using \eqref{eq:L2estim}, it is obtained that 
\begin{align*}
      \sup_{0 \leq t \leq T-h}  \int_{\R^n} |\Psi^h_i|^2 \, \d w_i \bigg |_t +  \int_{T'}^T \int_{\R^n} |\nabla \Psi^h_i|^2 \, \d w_i \d t &\leq  \int_{\R^n} |\Phi_i|^2 \, \d w_i \bigg |_{T-h}\\& \leq \int_{T-h}^T \int_{\R^n} |\nabla f|^2 \, \d w_i \d t.
\end{align*}
Hence, by Cauchy-Schwarz inequality, we have 
\begin{align*}
    \biggl |\int_0^{T-h} \int_{B} |\Psi^h_i| \, \d x \d t \biggr |^2 &\leq T w_i^{-1}(B) \int_0^{T-h} \int_{B} |\Psi^h_i|^2 \, \d w_i \d t\\ & \leq  T^2 w_i^{-1}(B) \int_{T-h}^T \int_{\R^n} |\nabla f|^2 \, \d w_i \d t,
\end{align*}
for every ball $B \subset \R^n.$
Therefore, $$\lim_{h \to 0}\limsup_{i \to \infty} \|\tau_h \Phi_i - \Phi_i\|_{\L^1([0,T-h] \times B)}=0,$$ for every ball $B \subset \R^n$. Moreover, by using Cauchy-Schwarz inequality and \eqref{eq:L2estim}, we deduce that 
\begin{align*}
    \biggl ( \int_{B}  \biggl|\nabla \int_{t_1}^{t_2} \Phi_i \, \d t \biggr| \, \d x \biggr)^2 & \leq  (t_2-t_1) \, w_i^{-1}(B)\int_{B}  \int_{t_1}^{t_2} |\nabla \Phi_i|^2 \, \d t  \d w_i \\& \leq  T \, w_i^{-1}(B) \int_{B}  \int_{t_1}^{t_2} |\nabla f|^2 \, \d t  \d w_i
\end{align*}
for every $0 \leq t_1 \leq t_2 \leq T$ and ball $B \subset \R^n.$ Then, for every $0 \leq t_1 \leq t_2 \leq T$, by Sobolev embedding, $\int_{t_1}^{t_2} \Phi_i \, \d t $ is relatively compact in $\L^1_{\loc}(\R^n)$. Hence, by the classical result in \cite[Thm. 1]{Simon}, $\partial_t \phi_i = \Phi_i + f$ is relatively compact in $C([0,T];\L^1_{\loc}(\R^n))$. Then, since $\phi_i$ converges to $\phi$ in $C([0,T];\L^1_{\loc}(\R^n))$ by Proposition \ref{prop:advecsol}, we obtain that $\partial_t \phi_i$ converges to $\partial_t \phi$ in $C([0,T];\L^1_{\loc}(\R^n))$. This together with \eqref{eq:timederivendpoin} and \eqref{eq:Linfboutimederiv} completes the proof.
    
\end{proof}

We also need the following $\L^1$-estimate lemma. \begin{lem}
\label{lem:upperL1bound}
    Let $f \in C_0([0,T] \times \R^n)$ and $\phi$ be the adjoint convection solution of \eqref{eq:dualsol2} in $[0,T] \times \R^n$. Then, 
\begin{align*}
    \sup_{0\leq t \leq T}\int_{\R^n} |\phi| \, \d w \bigg |_t \leq \int_0^T \int_{\R^n} |f| \, \d w \d t.
\end{align*}
    Moreover, if $w \in \L^1(\R^n)$, then
    \begin{align*}
         \int_{\R^n} \phi \, \d w \bigg |_s = \int_s^T \int_{\R^n} f \, \d w \d t,
    \end{align*}
    for every $0 \leq s \leq T.$
\end{lem}
\begin{proof}
   Without loss of generality, by decomposing $f = f^+ - f^-$, we can assume that $f$ is non-negative. Let $\phi_i$ be the approximating sequence in Lemma \ref{lem:approxpoten} and consider the test functions $\psi_j \, w_i$ where $\psi_j$ is defined in \eqref{eq:spectestfunc}. Then, by Lemma \ref{lem:approxpoten}, 
 $0 \leq \phi_i \leq T \|f\|_{\L^{\infty}([0,T] \times \R^n)}$ and
\begin{align*}
  \int_{\R^n} \phi_i \, \psi_j \, \d w_i \bigg |_t + \int_t^T \int_{\R^n} -\phi_i \, \div (w_i \nabla \psi_j) \, \d x \d t = \int_t^T \int_{\R^n} f \, \psi_j \d w_i \d t.
\end{align*}
Now, by Hölder's inequality and Lemma \ref{prop:advecsol}, we obtain
\begin{align*}
    &\sup_{0 \leq t \leq T}  \biggl| \int_t^T \int_{\R^n} -\phi_i \, \div (w_i \nabla \psi_j) \, \d x \d t\biggr| \\&\leq \|w_i\, \phi_i\|_{\L^p([0,T] \times \R^n)} \|w_i^{-1}\div(w_i \nabla \psi_j )\|_{\L^{p'}([0,T] \times \R^n)} 
    \\& \leq e^{T}\sup_{\R^n} w_i^{\frac{p-1}{p}} \|f\|_{\L^p_{w_i}([0,T] \times \R^n)}  \|w_i^{-1}\div(w_i \nabla \psi_j )\|_{\L^{p'}([0,T] \times \R^n)},
\end{align*}
for every $p>1.$ Letting $p <\frac{n}{n-1}$ and taking $j \to \infty$ in the inequality above, we arrive at
\begin{align*}
       \int_{\R^n} \phi_i  \, \d w_i  \bigg |_s =\int_s^T \int_{\R^n} f \, \d w_i \d t,
\end{align*}
for every $s \in [0,T],$ where we used Lemma \ref{lem:approxpoten} to derive
\begin{align*}
    &\sup_{\R^n} w_i = e^{-\inf_{\R^n} V_i}< \infty, \\
    &\sup_{\R^n}w_i^{-1}|\nabla w_i |
    \leq \|V_i\|_{C^{0,1}(\R^n)} < \infty.
    \end{align*}
   Now, using Lemma \ref{lem:approxpoten} and Proposition \ref{prop:advecsol}, $\phi_i(s,\cdot)$ converges pointwise to $\phi(s,\cdot)$ for every $s \in [0,T],$ up to a subsequence depending on $s$, and $w_i$ converges to $w$ pointwise and in $\L^1_{\loc}(\R^n)$. Hence, by Fatou's lemma and Lebesgue's dominated convergence, it is obtained that 
\begin{align*}
       \int_{\R^n} \phi  \, \d w  \bigg |_s \leq \int_s^T \int_{\R^n} f \, \d w \d t.
\end{align*}
for every $0 \leq s \leq T$. Moreover, by Lemma \ref{lem:approxpoten}, Proposition \ref{prop:advecsol}, and Lebesgue's dominated convergence, the inequality above is equality if $w \in \L^1(\R^n).$

\end{proof}

Now, we define the notion of convection solutions.
\begin{defn}
\label{def:convec2}
    We say that $\mu \in \mathcal{M}([0,T] \times \R^n)$ is a  convection solution for 
\begin{align}
\label{eq:convecsol2}
\partial_t \mu - \Delta \mu - \div ( \mu \nabla V ) = R,
\end{align}
    with initial value $g$ if 
\begin{align*}
    \int_{0}^T \int_{\R^n} f \, \d \mu = \int_{\R^n} \phi \, g \, \d x \bigg |_0 + \int_0^T \int_{\R^n} \phi \, R \,  \d x \d t,
\end{align*}
for every $f \in C^{\infty}_0([0,T] \times \R^n)$ and weak solution $\phi$ of \eqref{eq:dualsol2}. We say that a convection solution $\mu$ for
    \begin{align*}
        \partial_t \mu - \Delta \mu - \div(\mu \, \nabla V) = R,
    \end{align*}
    with the initial value $g$ satisfies the conservation of mass if the measure $\nu(I) := \mu(I \times \R^n)$, for Borel subsets $I \subset [0,T]$, is absolutely continuous with respect to Lebesgue's measure and
\begin{align*}
    \frac{\d \nu}{\d \mathcal{L}^1}(s)= \int_{\R^n} g \, \d x + \int_0^s \int_{\R^n} R \, \d x \d t,
\end{align*}
for a.e. $s \in [0,T].$
    
\end{defn}

\begin{proof}[Proof of Theorem \ref{thm:potentialdrift}]
We prove the Theorem in several steps.\\
\textbf{1}. Similar to Theorem \ref{thm:boundeddrift}, the uniqueness follows directly by the definition. To prove the existence, by the linearity of the equation, we can assume that $g, R$ are non-negative. Let $g_i,R_i$ be the non-negative smooth approximation of $g,R$ in $\L^1(\R^n), \L^1([0,T] \times \R^n)$ constructed in Theorem \ref{thm:boundeddrift}. By a standard theory of parabolic equations, see \cite{A}, there exists a non-negative weak solution $\mu_i \in C^{\infty}([0,T] \times \R^n) \cap \L^2([0,T]; \H^1(\R^n))$ of
\begin{align*}
    \partial_t \mu_i - \Delta \mu_i - \div(\mu_i \nabla V_i )  = R_i.
\end{align*}
with $\mu_i(0,x)=g_i(x)$ for every $x \in \R^n.$ By Theorem \ref{thm:boundeddrift}, we have
\begin{align*}
    \mu_i([0,T] \times \R^n) \leq T\int_{\R^n} g_i \, \d x + T \int_{0}^T \int_{\R^n} R_i \, \d x \d t,
\end{align*}
and 

\begin{align*}
 \int_{0}^T \int_{\R^n} f \, \d \mu_i = \int_{\R^n} \phi_i \, g_i \, \d x \bigg |_0 + \int_0^T \int_{\R^n} \phi_i \, R_i \,  \d x \d t,
\end{align*}
for every $f \in C_0^{\infty}([0,T] \times \R^n)$, where $\phi_i \in C^{\infty}([0,T] \times \R^n) \cap \L^2([0,T];\H^1(\R^n))$ is the weak solution of
\begin{align*}
    - \partial_t \phi_i - \Delta \phi_i + \nabla \phi_i \cdot \nabla V_i = f,
\end{align*}
in $[0,T] \times \R^n$ with $\phi(T,x)=0$ for $x \in \R^n.$ Hence, up to a subsequence, $\mu_i$ converges to $\mu \in \mathcal{M}([0,T] \times \R^n)$ in the vague topology, see \cite[Thm. 1.41]{EG}. Moreover, by Lemma \ref{lem:approxpoten}, $\phi_i$ is uniformly bounded and converges to $\phi \in \L^2_{w}([0,T];\H^1_w(\R^n)) \cap C([0,T];\L^2_{w}(\R^n))$ in $C([0,T];\L^1_{\loc}(\R^n))$, where $\phi$ is the unique adjoint convection solution of \eqref{eq:dualsol2}. In conclusion, 
\begin{align}
\label{eq:convecsoles}
 \int_{0}^T \int_{\R^n} f \, \d \mu = \int_{\R^n} \phi \, g \, \d x \bigg |_0 + \int_0^T \int_{\R^n} \phi \, R \,  \d x \d t,
\end{align}
for every $f \in C_0^{\infty}([0,T] \times \R^n)$, where $\phi$ is the adjoint convection solution of \eqref{eq:dualsol2}, and 
\begin{align}
\label{eq:upperbound}
 |\mu|([0,T] \times \R^n) \leq  T\int_{\R^n} g \, \d x \bigg |_0 + T \int_0^T \int_{\R^n}  R \,  \d x \d t.
\end{align}
Moreover, 
\begin{align*}
   \lim_{i \to \infty} \int_0^T \int_{\R^n} f \, \d \mu_i = \int_{\R^n} \phi \, g \, \d x \bigg |_0 + \int_0^T \int_{\R^n} \phi_i \, R \,  \d x \d t =\int_0^T \int_{\R^n} f \, \d \mu.
\end{align*}
Thus, $\mu_i$ converges vaguely to $\mu$ without any need to pass to a subsequence.\\
\textbf{2}.
Assume that \begin{align*}
 \int_{\R^n} |g|^{q'} w^{-\frac{q'}{q}} \, \d x   + \int_0^T \biggl(\int_{\R^n} |R|^{p'} w^{-\frac{p'}{p}} \, \d x\biggr)^{\frac{1}{p'}} \d t < \infty,
\end{align*}
for some $p>1,q >1$. To show that $\mu \in \L^1([0,T] \times \R^n)$, by Radon-Nikodym, it is enough to prove that $\mu(E)=0$ for every subset $E \subset [0,T] \times \R^n$ of Lebesgue measure zero. Consider a subset $E \subset [0,T] \times \R^n$ of Lebesgue measure zero and functions $f_i \in C^{\infty}_0([0,T] \times \R^n)$ approximating $1_E$ in $\L^{\infty}([0,T] \times \R^n)$-weak*. By Lemma \ref{lem:approxpoten}, there exists an adjoint convection solution $\gamma_i \in  \L^2([0,T];\H^1_w(\R^n)) \cap C([0,T];\L^2_{w}(\R^n))$ of 
\begin{align*}
    -\partial_t \gamma_i -e^{V} \div( e^{-V} \nabla \gamma_i)=f_i,  
\end{align*}
in $[0,T] \times \R^n$. Hence,  
\begin{align*}
       \int_{0}^T \int_{\R^n}  f_i \, \d \mu = \int_{\R^n} \gamma_i \, g \, \d x \bigg |_0 + \int_0^T \int_{\R^n} \gamma_i \, R \,  \d x \d t.
\end{align*}
Then, by Lemma \ref{lem:approxpoten} and Hölder's inequality, we have
\begin{align*}
 \biggl| \int_0^T \int_{\R^n} f_i \, \d \mu \biggr| &\leq \|\gamma_i(0,\cdot)\|_{\L^q_w( \R^n)} \, \|g\, w^{-1/q}\|_{\L^{q'}(\R^n)} \\& +T^{1/p} \sup_{0 \leq t \leq T}  \|\gamma_i(t,\cdot) \|_{\L^p_w(\R^n)} \, \|R \, w^{-1/p}\|_{\L^1([0,T];\L^{p'}(\R^n))}
  \\ & \leq e^T \|f_i\|_{\L_w^q([0,T] \times \R^n)} \|g\, w^{-1/q}\|_{\L^{q'}(\R^n)}\\& +T^{1/p} e^T \|f_i\|_{\L^p_w([0,T] \times \R^n)}  \|R \, w^{-1/p}\|_{\L^1([0,T];\L^{p'}(\R^n))}.
\end{align*}
Taking $i \to \infty$ implies that $\mu(E)=0.$\\
\textbf{3}. In this part, we derive an upper bound estimate for the convection solution. Let $R,g$ satisfy \eqref{eq:boundednes}. Without loss of generality, again we assume that $g,R$ are non-negative and $\mu$ is the convection solution of \eqref{eq:convecsol2} with initial data $g.$ Let $s \in (0,T)$ and $x \in \R^n$. We consider a sequence of function $f_i \in C^{\infty}_0(
   (1+\frac{1}{i})I \times (1+\frac{1}{i})B)$ satisfying $f_i =1$ in $I \times B$, where $I \subset [0,T], B \subset \R^n$ are open balls satisfying  $s \in I, x \in  B$, and 
$i$ is large enough such that $ (1+\frac{1}{i}) I \subset [0,T]$. Then,
   \begin{align*}
       \int_{0}^T \int_{\R^n} \frac{1}{|I \times B|} f_i \, \d \mu = \int_{\R^n} \psi_i \, g  \, \d x + \int_{0}^T \int_{\R^n} \psi_i\,  R  \, \d x \d t,
   \end{align*}
where $\psi_i \in \L^2_{w}([0,T];\H^1_w(\R^n)) \cap C([0,T];\L^2_{w}(\R^n))$ is the non-negative adjoint convection solution of 
\begin{align*}
      -\partial_t \psi_i - e^V  \div(e^{-V} \nabla \psi_i) =\frac{1}{|I \times B|}f_i,
\end{align*}
in $[0,T] \times \R^n$. Hence, by Lemma \ref{lem:upperL1bound}, it is concluded that
\begin{align*}
\frac{\mu(I \times B)}{|I \times B|}  \leq  \frac{1}{|I \times B|}  \int_{0}^T \int_{\R^n}f_i \, \d \mu &= \frac{1}{|I \times B|} \biggl(\int_{\R^n} \psi_i \, g  \, \d x + \int_{0}^T \int_{\R^n} \psi_i\,  R  \, \d x \d t\biggr)  \\ &\leq \frac{1}{|I \times B|}  \|g\,  w^{-1}\|_{\L^{\infty}} \int_0^T \int_{\R^n} f_i \, \d w \d t \\ &+ \frac{T}{|I \times B|} \|R \, w^{-1}\|_{\L^{\infty}} \int_{0}^T \int_{\R^n} f_i \, \d w \d t.
\end{align*}
Hence, by \cite[Lemma 1.2]{EG}, Radon-Nikodym theorem, and Step 1, it is concluded that $\mu \in \L^1([0,T] \times \R^n)$ and 
\begin{align*}
   \|\mu \, w^{-1}\|_{\L^{\infty}} \leq \|g \,  w^{-1}\|_{\L^{\infty}}+ T \|R \, w^{-1}\|_{\L^{\infty}}.  
\end{align*}
To prove the strong maximum principle, without loss of generality, let $R,g$ be non-negative and $\inf g \, w^{-1} +\inf R \, w^{-1}>0$. Then, $w \in \L^1(\R^n)$ and, by Lemma \ref{lem:upperL1bound}, we have 
\begin{align*}
    &\frac{\mu(\big(1+\frac{1}{i}\big)I \times \big(1+\frac{1}{i}\big)B)}{\big|\big(1+\frac{1}{i}\big)I \times \big(1+\frac{1}{i}\big)B\big|}  \geq  \frac{1}{\big|\big(1+\frac{1}{i}\big)I \times \big(1+\frac{1}{i}\big)B\big|}  \int_{0}^T \int_{\R^n}f_i \, \d \mu \\  &= \frac{1}{|I \times B|} \biggl(\int_{\R^n} \psi_i \, g  \, \d x + \int_{0}^T \int_{\R^n} \psi_i\,  R  \, \d x \d t\biggr)  
    \\&\geq  \frac{\inf g \, w^{-1}}{\big|\big(1+\frac{1}{i}\big)I \times \big(1+\frac{1}{i}\big)B\big|}  \, \int_{0}^T \int_{\R^n} f_i \, \d w \d t
    \\ &  + \frac{\inf R \, w^{-1}}{\big|\big(1+\frac{1}{i}\big)I \times \big(1+\frac{1}{i}\big)B\big|}\int_{0}^T \int_t^T \int_{\R^n} f_i \, \d w \d r \d t.
\end{align*}
This concludes that
\begin{align*}
     \liminf_{r \to 0} \frac{\mu((s-r,s+r) \times B(x,r))}{|(s-r,s+r) \times B(x,r)|}  \geq \inf (g\, w^{-1}  + s \inf R\, w^{-1})\,  w(x),
\end{align*}
for every $0<s<T$ and Lebesgue point $x \in \R^n$ of $w.$\\
\textbf{4}. Now, we prove the conservation of mass. Assume that $w \in \L^1(\R^n)$ and $R,g$ are non-negative functions. We take the same sequences of functions $R_i,g_i,V_i,\mu_i$ as in Step 1. Define $f=1_{\R^n \setminus B}$ for a ball $B \subset \R^n$. Then, by Lemma \ref{lem:approxpoten}, $w_i$ converges to $w$ in $\L^1(\R^n)$ and $f \in \L^{\infty}([0,T] \times \R^n) \cap \L^2_{w_i}([0,T] \times \R^n)$ for every $i$. Hence, by Proposition \ref{prop:advecsol} and Lemma \ref{lem:upperL1bound}, there exists an adjoint convection solution $\phi^B_i \in C([0,T]; \L^2_w(\R^n)) \cap \L^2([0,T]; \H^1_w(\R^n))$ of \begin{align*}
      -\partial_t \phi^B_i - e^{V_i}\div (e^{-V_i}\, \nabla \phi^B_i ) =f,  
\end{align*}
in $[0,T] \times \R^n$, where 
\begin{align*}
    \|\phi^B_i\|_{\L^{\infty}(\R^n)} \leq T, \,  \sup_{0 \leq t \leq T}\|\phi^B_i(t,\cdot)\|_{\L_{w_i}^{1}(\R^n)} \leq w_i(\R^N \setminus B). 
\end{align*}
Moreover, as $i \to \infty$, $\phi_i^B$ converges to $\phi^B$ in $C([0,T]; \L^1_{\loc}(\R^n))$. Hence, by Fatou's lemma and Lebesgues's dominated convergence, we derive
\begin{align*}
\sup_{0\leq t \leq T}\|\phi^B(t,\cdot)\|_{\L_{w}^{1}(\R^n)} &\leq w(\R^N \setminus B), \\
 \limsup_{i \to \infty}  \int_0^T \int_{\R^n \setminus B} \, \d \mu_i &\leq \int_{\R^n}  \phi^B \, g \, \d x \bigg |_0 + \int_{0}^T \int_{\R^n}  \phi^B \, R \, \d x \d t.
\end{align*}
In conclusion, by replacing $B$ by $B(0,j)$ and using again Proposition \ref{prop:advecsol}, we derive that, $\phi^{B(0,j)}$ is uniformly bounded and converges to zero in $C([0,T];\L^1_{\loc}(\R^n))$. Then, \begin{align*}
  \lim_{j \to \infty}  \limsup_{i \to \infty}  \int_0^T \int_{\R^n \setminus B(0,j)} \, \d \mu_i =0,
\end{align*}   
where we used Lebesgue's dominated convergence above. Hence, $\mu_i$ is a tight sequence, and, by Prokhorov's theorem and Step 1, it converges weakly to $\mu$. Let $\Phi \in C^1([0,T])$ be a fixed function where $\Phi(T) =0$. Then, $-\partial_t \Phi - \Delta \Phi + \nabla \Phi \cdot \nabla V_i = -\partial_t \Phi$ weakly in $[0,T] \times \R^n$ and
\begin{align*}
   \int_0^T \int_{\R^n}  - \partial_t \Phi \, \mu_i  \, \d t  \d x =    \int_{\R^n}\Phi  \, g_i \, \d x + \int_0^T \int_{\R^n}\Phi \, R_i \, \d x \d t.
\end{align*}
Taking $i \to \infty$ and using the weak convergence of $\mu_i$ to $\mu$, we arrive at 
\begin{align*}
     \int_0^T \int_{\R^n}  -  \partial_t \Phi \, \d \mu =    \int_{\R^n} \Phi  \, g \, \d x + \int_0^T \int_{\R^n}\Phi \, R  \, \d x \d t.
\end{align*}
Let $s \in (0,T)$ be a fixed point, $i > \max(\frac{1}{T-s}, \frac{1}{s})$, and $\varphi_{\epsilon}$ is an approximation of the identity. Setting $\Phi_{\epsilon}(t) = \frac{i}{2} \int_{t}^T \varphi_{\epsilon} \ast 1_{[s-\frac{1}{i}, s+\frac{1}{i}]}  \, \d t$ for every $t \in [0,T],$ we arrive at
\begin{align*}
     \int_0^T \int_{\R^n}  -  \partial_t \Phi_{\epsilon} \, \d \mu =   \int_{\R^n}\Phi_{\epsilon} \, g \, \d x \bigg |_0 + \int_0^T \int_{\R^n} \Phi_{\epsilon} \, R\, \d x \d t.
\end{align*}
Taking $\epsilon \to 0$ implies that \begin{align*}
   \frac{i}{2} \int_{s-\frac{1}{i}}^{s+\frac{1}{i}} \int_{\R^n} \, \d \mu = \int_{\R^n} g \, \d x + \int_{0}^{s-\frac{1}{i}} \int_{\R^n} R\, \d t \d x+\frac{i}{2} \int_{s-\frac{1}{i}}^{s+\frac{1}{i}} \int_{\R^n}\biggl(\frac{i\, s+1-i\, t}{2}\biggr) R \, \d r \d t \d x.
\end{align*} 
Hence, by \cite[Lem. 1.2]{EG}, Radon-Nikodym theorem, and letting $i \to \infty$, we arrive at the signed measure $\nu(I) := \mu(I \times \R^n)$, for Borel subsets $I \subset [0,T]$, is absolutely continuous with respect to Lebesgue's measure and
\begin{align*}
 \frac{\d \nu}{\d \mathcal{L}^1}(s) =   \int_{\R^n} g \, \d x + \int_0^s \int_{\R^n} R  \, \d x \d t
\end{align*}
 for a.e. $s \in [0,T].$\\
 \textbf{5}. Finally, we prove the continuous to the initial data. Without loss of generality, we assume that $R,g$ are non-negative functions. Let $f \in C^{\infty}_0(\R^n)$ and $\phi_i \in \L^2([0,T];\H^1(\R^n))$ be the smooth weak solution of
 \begin{align*}
     -\partial_t \phi_i - \Delta \phi_i + \nabla V_i \cdot \nabla \phi_i = f,
 \end{align*}
in $[0,T] \times \R^n$ with $\phi_i(T,x)=0$ for every $x \in \R^n,$ and $g_i,R_i, \mu_i$ be the sequences defined in Step 1. Then, by Lemma \ref{lem:timederiv}, we have
\begin{equation*}
 \int_{0}^s \int_{\R^n} f \, \mu_i \, \d x \d t= \int_{\R^n} \phi_i^s \, g_i \, \d x \bigg |_0 + \int_0^s \int_{\R^n} \phi_i^s \, R_i \,  \d x \d t, 
 \end{equation*}
 where $s \in (0,T), \phi_i^s(t,x) := \phi_i(t+T-s,x)$ for every $ 0 \leq t \leq s , x \in \R^n$. Let $0 < s < T$ be fixed. Now, using $\phi_i^s(s,x)=\phi_i(T,x)=0$ for every $x \in \R^n$ and taking the derivative of both sides of above equality with respect to $s$, we obtain
\begin{align*}
     \int_{\R^n} f \, \mu_i \, \d x \bigg |_s = \int_{\R^n} -\partial_t \phi_i \, g_i \, \d x \bigg |_{T-s} + \int_0^s \int_{\R^n} -\partial_t \phi_i(t+T-s,\cdot)\, R_i \, \d x \d t.
\end{align*}
Hence, by Lemma \ref{lem:timederiv}, it is concluded that
\begin{align*}
   \biggl | \int_{\R^n} f \, \mu_i \, \d x \bigg |_s - \int_{\R^n} f\, g \, \d x \biggr | &\leq \biggl |  \int_{\R^n} -\partial_t \phi_i \, g_i \, \d x \bigg |_{T-s} - \int_{\R^n} f \, g \, \d x \biggr |\\& + \|f\|_{\L^{\infty}} \int_0^s \int_{\R^n} R_i \, \d x \d t.
\end{align*}
 Then, by vague convergence of $\mu_i$ to $\mu$ proved in Step 1, Lemma \ref{lem:timederiv}, and Lebesgue's dominated convergence, we have
\begin{align*}
  \biggl | \frac{1}{h} \int_s^{s+h}\int_{\R^n} f \, d \mu - \int_{\R^n} f\, g \, \d x \biggr | &= \lim_{i \to \infty} \biggl | \frac{1}{h} \int_{s}^{s+h} \int_{\R^n} f \, \mu_i \, \d x \d t  - \int_{\R^n} f\, g \, \d x \biggr | \\ & \leq \frac{1}{h}\int_s^{s+h}\biggl |  \int_{\R^n} -\partial_t \phi \, g \, \d x \bigg |_{T-t} - \int_{\R^n} f \, g \, \d x \biggr | \, \d t\\& + \|f\|_{\L^{\infty}} \frac{1}{h}\int_s^{s+h} \int_0^t \int_{\R^n} R \, \d x \d r \d t,
\end{align*}
for every $0<h < T-s$. In conclusion, by using Lemma \ref{lem:timederiv} and Lebesgue's dominated convergence again, we arrive at
\begin{align*}
    \lim_{(h,s) \to 0} \biggl | \frac{1}{h} \int_s^{s+h} \int_{\R^n} f \, d \mu - \int_{\R^n} f\, g \, \d x \biggr |=0,
\end{align*}
where the limit is taken in the domain $0<h+s <T$.

\end{proof}

\begin{exm}
    If $R=0, g =  e^{-V}$, then by Theorem \ref{thm:potentialdrift}, the convection solution $\mu$ of
    \begin{align}
    \label{eq:stableconv}
    \partial_t \mu -\Delta \mu - \div(\mu \nabla V)=0,    
    \end{align}
     in $[0,T] \times \R^n$ with initial data $g$ satisfies $\mu \in \L^1([0,T] \times  \R^n)$ and  
\begin{align*}
  1 \leq \mu \, e^{V} \leq 1
\end{align*}
for a.e. $(t,x) \in [0,T] \times \R^n$. Hence, $\mu = e^{-V}$ which is the stable solution of \eqref{eq:stableconv}.

\end{exm}

\begin{rem}
    We remark that if $\nabla V \in \L^{\infty}(\R^n)$, equivalently $V$ is globally Lipschitz, then, by Proposition \ref{prop:dualsol} and Proposition \ref{prop:advecsol}, Definition \ref{def:dualsol} and Definition \ref{def:dualsolpoten} coincide. In conclusion, by Theorem \ref{thm:boundeddrift} and Theorem \ref{thm:potentialdrift}, the two definitions of convection solutions, namely Definition \ref{def:convec1} and Definition \ref{def:convec2}, are equivalent whenever $\nabla V \in \L^{\infty}(\R^n)$. Furthermore, with a slight change in the arguments, we can generalize the definition of convection-diffusion solutions and Theorem \ref{thm:potentialdrift} to $\partial_t \mu -\div A \nabla \mu - \div(\mu \nabla V)$, where $A$ is a positive definitive matrix with constant coefficients.
\end{rem}

\end{document}